\documentclass[a4paper,10pt]{amsart}
\usepackage[utf8]{inputenc}
\usepackage[T1]{fontenc}
\usepackage{longtable}
\usepackage{graphicx}
\usepackage{subcaption}
\usepackage{mathtools}
\usepackage{color}
\usepackage{enumitem}

\usepackage{cite}
\usepackage{url}
\usepackage{amsmath}
\usepackage{amssymb}
\usepackage{amsthm}
\usepackage{setspace}
\usepackage[english]{babel}
\usepackage{mathrsfs}
\usepackage{mathtools}
\usepackage{xcolor}
\usepackage{mathrsfs}
\usepackage{hyperref} 
\newcommand{\N}{\mathbb{N}}
\newcommand{\C}{\mathbb{C}}
\newcommand{\R}{\mathbb{R}}

\newcommand{\dd}{\mathrm{d}}
\newcommand{\dz}{\partial_z}
\newcommand{\dzbar}{\partial_{\bar{z}}}
\newcommand{\bn}{\mathbf{n}}

\newtheorem{theorem}{Theorem}[section]
\newtheorem{lemma}[theorem]{Lemma}
\newtheorem{corollary}[theorem]{Corollary}
\newtheorem{proposition}[theorem]{Proposition}

\newtheorem{definition}[theorem]{Definition}
\theoremstyle{definition}
\newtheorem{remark}[theorem]{Remark}

\renewcommand{\epsilon}{\varepsilon}

\newcommand{\ee}{\mathrm{e}}

\DeclareMathOperator{\Dom}{Dom}
\DeclareMathOperator{\Ran}{Ran}
\DeclareMathOperator{\Ker}{Ker}

\DeclareMathOperator{\diag}{diag}

\title[Schrödinger operators with~non-local singular potentials]{Two-dimensional Schrödinger operators with~non-local singular potentials}

\date{\today}

\begin{document}

\author[L. Heriban]{Luk\'{a}\v{s} Heriban}
\address{Department of Mathematics, Faculty of Nuclear Sciences and Physical Engineering\\
Czech Technical University in Prague \\
Trojanova 13, 120 00, Prague \\
E-mail: {\tt heribluk@fjfi.cvut.cz}}

\author[M. Holzmann]{Markus Holzmann}
\address{Institut f\"{u}r Angewandte Mathematik\\
Technische Universit\"{a}t Graz\\
 Steyrergasse 30, 8010 Graz, Austria\\
E-mail: {\tt holzmann@math.tugraz.at}}

\author[C. Stelzer-Landauer]{Christian Stelzer-Landauer}
\address{Institut f\"{u}r Angewandte Mathematik\\
Technische Universit\"{a}t Graz\\
 Steyrergasse 30, 8010 Graz, Austria\\
E-mail: {\tt christian.stelzer@tugraz.at}}

\author[G. Stenzel]{Georg Stenzel}
\address{Institut f\"{u}r Angewandte Mathematik\\
Technische Universit\"{a}t Graz\\
 Steyrergasse 30, 8010 Graz, Austria\\
E-mail: {\tt gstenzel@math.tugraz.at}}

\author[M. Tu\v{s}ek]{Mat\v{e}j Tu\v{s}ek}
\address{Department of Mathematics, Faculty of Nuclear Sciences and Physical Engineering\\
Czech Technical University in Prague \\
Trojanova 13, 120 00, Prague \\
E-mail: {\tt matej.tusek@fjfi.cvut.cz}}

\keywords{Schr\"odinger operator, non-local singular potentials, extension theory, boundary triples, spectral theory, non-relativistic limit}

\begin{abstract}
In this paper we introduce and study a family of self-adjoint realizations of the Laplacian in $L^2(\R^2)$ with a new type of transmission conditions along a closed bi-Lipschitz curve $\Sigma$. These conditions incorporate jumps in the Dirichlet traces both  of the functions in the operator domains and of their Wirtinger derivatives and are non-local. Constructing a convenient generalized boundary triple, they may be parametrized by all compact hermitian operators in $L^2(\Sigma;\C^2)$. Whereas for all choices of parameters the essential spectrum is stable and equal  to $[0, +\infty)$, the discrete spectrum exhibits diverse behaviour. While in many cases it is finite, we will describe also a class of parameters for which the discrete spectrum is infinite and accumulates at $-\infty$. The latter class contains a non-local version of the oblique transmission conditions. Finally, we will connect the current model to its relativistic counterpart studied recently in \cite{HT23}.
\end{abstract}

\subjclass[2020]{Primary 81Q10; Secondary 35Q40}
\maketitle

\section{Introduction}

Differential operators with singular potentials supported on sets of Lebesgue measure zero have attracted a lot of attention in the mathematical community in the recent years; cf. the monograph \cite{AGHH05} or the review papers \cite{BHSLS23, E08} and the references therein. In many situations, there is a physical motivation to study these operators, as the singular interactions can be viewed as idealized replacements for regular interactions--let us just mention \cite{dKP1931, RuSch53, T35} as representatives of pioneering influential works. However, there is often also a pure mathematical motivation as these models have interesting spectral properties, see \cite{AMV15, BHOP20, BHSLS23, BHS23, BHT22, BLL12, BP24, EI01, EP14, MPS16, S04a, S04b} for a small and certainly incomplete selection of results, and the idealized models are used as a playground to see what one can expect from more regular models.

The main aim of this paper is to introduce and study a new family of operators, which  are self-adjoint in $L^2(\R^2)$ and act as the two-dimensional Laplacian away from the Lipschitz continuous boundary $\Sigma$ of an open bounded set $\Omega_+$. In particular, this family contains operators that may be naturally associated with the formal expression
\begin{equation} \label{eq:elementary_formal}
-\Delta+\alpha|\delta_\Sigma\rangle\langle\delta_\Sigma|+\beta |\delta_\Sigma\rangle\langle\dz\delta_\Sigma|+\bar\beta|\dz\delta_\Sigma\rangle\langle\delta_\Sigma|+\gamma|\dz\delta_\Sigma\rangle\langle\dz\delta_\Sigma|,
\end{equation}
where $\alpha,\,\gamma\in\R,\, \beta\in\C$, $\delta_\Sigma$ stands for the single layer (or $\delta$-) distribution supported on $\Sigma$, and $\partial_z$ denotes the Wirtinger derivative. Here, we are using the common physical \emph{bra-ket} notation extended to distributions which, in turn, are extended to functions that admit well-defined Dirichlet traces with respect to $\Omega_+$ and $\Omega_-:=\R^2\setminus\overline{\Omega_+}$. Following similar considerations as in Section 3 of \cite{HT23}, one derives that, for $f\equiv f_+\oplus f_-$ from a suitable subspace of $L^2(\R^2)\equiv L^2(\Omega_+)\oplus L^2(\Omega_-)$, the image of $f$ under \eqref{eq:elementary_formal} belongs to $L^2(\R^2)$ if and only if the transmission conditions 
\begin{equation} \label{eq:TC_sub}
4 \begin{pmatrix}
		\bar{\bn} ( \gamma_D^{+} \dzbar f_{+} - \gamma_D^{-} \dzbar f_{-} ) \\
		\bn (\gamma_D^{+} f_{+} - \gamma_D^{-} f_{-} ) 
 \end{pmatrix} 
 +B\, \begin{pmatrix}
				 \gamma_D^{+} f_{+} +  \gamma_D^{-} f_{-}   \\
				- \gamma_D^{+} \dzbar f_{+} - \gamma_D^{-} \dzbar f_{-} 
	  \end{pmatrix}=0
\end{equation}
are fulfilled,
where $\gamma_D^{\pm}$ denote the Dirichlet trace operators with respect to $\Omega_\pm$, $\bn:=n_1+i n_2$ is the complexified outer unit normal vector $n$ to $\Omega_+$, and $B$ is the finite rank operator in $L^2(\Sigma;\C^2)$ given by
\begin{equation} \label{def_B_intro}
B\varphi:=\begin{pmatrix}
			\alpha & \beta\\
			\bar{\beta} & \gamma
		  \end{pmatrix}
  \int_\Sigma\varphi.
\end{equation}
Notice that in \eqref{eq:TC_sub} the jumps of $f$ and its Wirtinger derivative are given by integral values of the traces of $f$ along $\Sigma$, i.e. the transmission condition is \emph{non-local}.

In our general setting, we will allow $B$ in \eqref{eq:TC_sub} to be any compact self-adjoint operator in $L^2(\Sigma;\C^2)$. The reason for this choice is essentially two-fold. Firstly, the transmission conditions \eqref{eq:TC_sub} then correspond to a generalization of the formal expression \eqref{eq:elementary_formal} that would look like
\begin{equation} \label{eq:compact_formal}
\begin{split}
- \Delta 
 + \sum_{n=1}^{\infty} b_n \Big( &| \varphi_{n}^1 \delta_{\Sigma} \rangle \langle \varphi_{n}^1 \delta_{\Sigma} | +   | \varphi_{n}^1 \delta_{\Sigma} \rangle \langle \dz (\varphi_{n}^2 \delta_{\Sigma}) | \\
 &+  | \dz (\varphi_{n}^2 \delta_{\Sigma}) \rangle \langle \varphi_{n}^1 \delta_{\Sigma} | + | \dz (\varphi_{n}^2 \delta_{\Sigma}) \rangle  \langle \dz (\varphi_{n}^2 \delta_{\Sigma}) | \Big),
\end{split}
\end{equation}
where $(b_n,(\varphi_n^1,\varphi_n^2)^T),\, n\in\N,$ are the eigenpairs of $B$ with orthonormal eigenfunctions; cf. Proposition~\ref{PropFormalDiffExpression}. Since the latter interaction is  a limit of linear combinations of non-local interaction terms that generalize those  in~\eqref{eq:elementary_formal}, we call also the more general interactions in~\eqref{eq:compact_formal} \emph{non-local}.
 The precise definition of the associated operator $T_B$ then reads
\begin{equation} \label{eq:def_TB_alt}
\begin{split}
T_B f &= ( - \Delta f_{+} ) \oplus ( - \Delta f_{-} ), \\
\text{Dom}(T_B) &= \big\{ f \in H^{1/2}(\mathbb{R}^2 \setminus \Sigma) \: \big| \: \Delta f_\pm \in L^2(\Omega_\pm), \partial_{\overline{z}} f_\pm \in H^{1/2}(\Omega_\pm) \text{ and \eqref{eq:TC_sub} holds}\big\},
\end{split}
\end{equation}
where $H^{1/2}$ denotes the $L^2$-based Sobolev spaces of order $\frac{1}{2}$.
With proper techniques, it is not difficult to show that $T_B$ is self-adjoint (see Theorem~\ref{TBSelfadjoint}). Our second  and most important motivation is that we found out that already in this setting, that comprises relatively small perturbations of the free Laplacian (in the sense that $T_0=-\Delta$ and that the potential in \eqref{eq:compact_formal} is a convergent sum of formal "finite rank operators"), the point spectrum of $T_B$ may be very diverse; the essential spectrum is, as expected, always stable and equal to $[0,+\infty)$. If $B$ has finite rank $k$ then $T_B$ has at most $k$ discrete eigenvalues. Also for any $B=B_{11}\oplus 0\equiv\diag(B_{11},0)$ in $L^2(\Sigma)\oplus L^2(\Sigma)\equiv L^2(\Sigma;\C^2)$, the discrete spectrum of $T_B$ is finite. Note that in that case, \eqref{eq:compact_formal} reduces to a Schr\"odinger operator with a $\delta$-potential with the "compact weight" $B_{11}$, 
\begin{equation} \label{eq:B_11_formal}
-\Delta+B_{11}\delta_\Sigma, 
\end{equation}
where the formal potential $B_{11}\delta_\Sigma$ maps a function $f$ with a well-defined Dirichlet trace to the distribution $(B_{11}\gamma_D f)\delta_\Sigma$.
 On the other hand, if the following condition
\begin{equation} \label{equation_S} \tag{S}
 B=\begin{pmatrix}
  B_{11} & B_{12} \\
  B_{21} & B_{22}
 \end{pmatrix} \geq 0 \text{ and } B_{22} \text{ has infinitely many positive eigenvalues}
\end{equation}
is satisfied,  then $\sigma_{\textup{disc}}(T_B)$ is infinite with the only accumulation point at $-\infty$, see Theorem~\ref{SpectrumTB} for details. 

Let us compare our model to others that were treated in the literature.
Allowing $B$ to be just hermitian (and not necessarily compact) one can reproduce the so-called $\delta$-shell and oblique transmission conditions. The former are usually given by 
\begin{equation} \label{eq:deltaTC}
\gamma_D^{+} f_{+}=\gamma_D^{-} f_{-}, \quad \gamma_N^{+} f_{+}-\gamma_N^{-} f_{-}=-\frac{\eta}{2}(\gamma_D^{+} f_{+}+\gamma_D^{-} f_{-}),
\end{equation}
where $\eta$ is a bounded real-valued function and $\gamma_N^\pm$ stand for the Neumann trace operators with respect to $\Omega_\pm$ and the normal vector $n$; cf. \cite{BEKS94,BLL12}. With these conditions one automatically has higher regularity of the functions in the domain of the associated operator $T_{\delta,\eta}$. Since, for functions obeying $\gamma_D^{+} f_{+}=\gamma_D^{-} f_{-}$, the left-hand side of the second condition in \eqref{eq:deltaTC} is equal to $2\bar{\bn} ( \gamma_D^{+} \dzbar f_{+} - \gamma_D^{-} \dzbar f_{-} )$, the operator $B$ corresponding to \eqref{eq:deltaTC} would be $\diag(\eta I, 0)$.  Let us emphasize that the natural formal expression for $T_{\delta,\eta}$ is $-\Delta+\eta \delta_\Sigma$, where $\eta \delta_\Sigma$ should be understood as a multiplication operator, i.e. it corresponds to $B_{11}=\eta$ in \eqref{eq:B_11_formal}. This is very different from \eqref{eq:elementary_formal} with $\beta=\gamma=0$, see also \eqref{def_B_intro}. It is known that $\sigma_{\textup{ess}}(T_{\delta,\eta})=[0,+\infty)$ and the negative spectrum is discrete and finite. In Theorem~\ref{SpectrumTB}~(iii) we will see that this picture is the same for compact coefficients $\eta = B_{11}$, i.e. for $T_B$ in~\eqref{eq:def_TB_alt} with $B = \diag(B_{11}, 0)$.

The oblique transmission conditions were introduced in \cite{BHS23} for $C^\infty$-smooth boundaries and studied only recently also for Lipschitz continuous ones \cite{BeCaPa24}. They can be written for $\eta \in \mathbb{R}$ as
\begin{equation} \label{eq:obliqueTC}
\gamma_D^{+}\dzbar f_{+}=\gamma_D^{-}\dzbar f_{-}, \quad \bn ( \gamma_D^{+}f_{+} - \gamma_D^{-}  f_{-} )=-\eta(\gamma_D^{+}\dzbar f_{+}+\gamma_D^{-}\dzbar f_{-}),
\end{equation}
which means that $B$ in \eqref{eq:TC_sub} would be $\diag(0, -4 \eta I)$. For the corresponding operator $T_{\angle,\eta}$, it was deduced that $\sigma_{\textup{ess}}(T_{\angle,\eta})=[0,+\infty)$. Moreover, $T_{\angle,\eta}$ is a non-negative operator when $\eta\geq 0$. On the other hand, if $\eta<0$, then $\sigma_{\textup{disc}}(T_{\angle,\eta})$ is infinite with the only accumulation point at $-\infty$. Therefore, in the latter case the spectrum of $T_{\angle,\eta}$ exhibits exactly the same unexpected behaviour as the spectrum of $T_B$ with $B$ compact and satisfying the assumption~\eqref{equation_S}. Recalling that this $T_B$ is a smaller perturbation of the free Laplacian than $T_{\angle,\eta}$, as the coefficient in the transmission conditions is compact, this effect is even more surprising in our setting.
We would like to point out that the oblique transmission conditions are different from those corresponding to $\delta'$-interactions as they are studied, e.g., in \cite{BLL12, MPS16}; see also \cite{ER16} for combinations and generalizations of $\delta$- and $\delta'$-interactions, which are also different from the interactions treated in this paper.

To analyze the operator $T_B$ we introduced a new generalized boundary triple $\{L^2(\Gamma;\C^2),\Gamma_0,\Gamma_1\}$ for $\overline{T}$, where
\begin{equation*}
\begin{split}
T f &= ( - \Delta f_{+} ) \oplus ( - \Delta f_{-} ), \\
\text{Dom}(T) &= \left\{ f \in H^{1/2}(\mathbb{R}^2 \setminus \Sigma) \: \big| \: \Delta f_\pm\in L^2(\Omega_\pm),\,\partial_{\overline{z}} f_\pm \in H^{1/2}(\Omega_\pm)\right\}.
\end{split}
\end{equation*}
The boundary mappings $(\Gamma_0,\Gamma_1)$ are chosen exactly in the way that $T_B=T\upharpoonright \Ker (\Gamma_0+B\Gamma_1)$.
We believe that our triple is of independent interest since it may be used to study transmission conditions with jumps in the Dirichlet traces both of the functions and of their Wirtinger derivatives, whereas other triples usually come with the normal instead of the Wirtinger derivative \cite{BLL12, BR15}. For a summary of basic results for generalized boundary triples in an abstract setting see Section~\ref{section_boundary_triples}. In particular, we will present a useful version of the Birman-Schwinger principle and a Krein type resolvent formula in the case when the parameter $B$ factorizes as $B=B_1 B_2$ with $B_2:\mathcal{G}\to\mathcal{K}$ and $B_1:\mathcal{K}\to\mathcal{G}$, where $\mathcal{G}$ is the boundary space and $\mathcal{K}$ is another (\emph{intermediate}) Hilbert space, which were, to the best of our knowledge, not contained in the literature; cf. Theorem \ref{GBT_TB_SA}. In particular, if $\dim(\mathcal{K})<+\infty$ then the Birman-Schwinger operator is also finite-dimensional.

Formal expressions similar to \eqref{eq:elementary_formal} were examined recently in the relativistic setting \cite{HT23}.  Putting the mass term equal to $1/2$ and introducing the speed of light $c>0$ they would look like
\begin{equation*}
	-ic\sum_{i=1}^2 \sigma_i \partial_i + \sigma_3 \frac{c^2}{2} + c|F\delta_\Sigma\rangle\langle G\delta_\Sigma|,
\end{equation*}
where $\{\sigma_i\}_{i=1}^3$ denote the usual Pauli matrices and $F,G\in L^2(\Sigma;\C^{2\times 2})$ are such that for all $f\in L^2(\Sigma;\C^2)$, $F\int_\Sigma (G^*f)=G\int_\Sigma (F^*f)$. The associated self-adjoint operator $D_{F,G}^c$ is defined properly in \eqref{DefNonLocalDirac}. We will look for its non-relativistic limit, i.e. we will study $\lim_{c\to+\infty}(D^c_{F,G}-c^2/2)$ in a well defined sense. In general, the non-relativistic limit provides a connection between the relativistic theory and the (usually well-studied) non-relativistic theory and thereby allows an interpretation of the relativistic model.  The procedure of the non-relativistic limit for the case of a regular perturbation is discussed in details in \cite[Chap. 6]{Thaller} and for usual (\emph{local}) $\delta$-point and $\delta$-shell interactions in \cite{BEHL18, BHS23, BD94,HT22}. There it was observed that to get a convergence as well as a limit which non-trivially depends on all interaction parameters, one has to scale these parameters appropriately. In our case, we will put $F_c:=S_cF$ and $G_c:=S_c G$ with $S_c:=\diag (c^{-1/2},c^{1/2})$. Then we will show in Theorem~\ref{theorem-nr lim} that, for every $w\in\C\setminus\R$,
\begin{equation*}
\left\lVert\left(D_{F_c,G_c}^c-w -\frac{c^2}{2} \right)^{-1}-(T_B-w)^{-1}\otimes \begin{pmatrix}
1 & 0\\
0 & 0
\end{pmatrix}\right\rVert = \mathcal{O}(c^{-1}),
\end{equation*}
where $B$ is the finite rank operator that acts as $Bf=VF\int_\Sigma (G^*V^* f)$ with $V:=\diag (1,-2i)$. In particular, this shows that the Schr\"odinger operators with non-local interactions in~\eqref{eq:elementary_formal} are indeed the non-relativistic counterpart of the relativistic models with non-local interactions studied in \cite{HT23}.

The paper is organized as follows. Section~\ref{section_preliminaries} is of preliminary nature. After collecting some notations in Section~\ref{section_notations}, Sobolev spaces and suitable trace theorems are discussed in Section~\ref{section_function_spaces}. Notations on distributions are summarized in Section~\ref{section_distributions}. In Section~\ref{section_int.oper.} some frequently used material on integral operators is presented, and in Section~\ref{section_boundary_triples} abstract definitions and results on generalized boundary triples are provided. Section~\ref{section_T_B} is devoted to the analysis of $T_B$ defined in~\eqref{eq:def_TB_alt}. In Section~\ref{section_GBT_T_B} a generalized boundary triple that is suitable for the study of $T_B$ is introduced. This is used in Section~\ref{section_def_selfadjointness} to show the self-adjointness of $T_B$, while the spectral properties of $T_B$ are studied in Section~\ref{section_spectrum}. Finally, in Section~\ref{section_nrl} we show that $T_B$ is the non-relativistic limit of Dirac operators with non-local singular interactions.

\section{Preliminaries} \label{section_preliminaries}

This section is devoted to some preliminary material that is used throughout the paper. First, in Section~\ref{section_notations} we collect some frequently used notations. Section~\ref{section_function_spaces} is devoted to function spaces and trace theorems, and in Section~\ref{section_distributions} we introduce some notations concerning distributions. In Section~\ref{section_int.oper.} we summarize some useful results for integral operators. Finally, in Section~\ref{section_boundary_triples} we give a summary about the theory of generalized boundary triples.

\subsection{Notation} \label{section_notations}
Let $\Omega_+$ be an open bounded set in $\R^2$ with at least Lipschitz continuous boundary $\partial\Omega_+ =\Sigma$ as described in Section~\ref{section_function_spaces}. The unit outer normal vector field and the unit tangent field along $\Sigma$ will be denoted by $n=(n_1,n_2)$ and $t=(t_1,t_2)$, respectively. We put $t=(-n_2,n_1)$, so that the basis $(n,t)$ has the standard orientation of $\R^2$, and $\Omega_-:=\R^2\setminus\overline{\Omega_+}$.   We denote the restriction of a function $f : \mathbb{R}^2 \to \mathbb{C}$ to $\Omega_\pm$ by $f_\pm$. The complexified vector $x=(x_1,x_2)\in\R^2$ will be denoted by $\mathbf{x}= x_1+i x_2\in\C$. Furthermore, the Wirtinger derivatives are given by $\partial_z  = \tfrac{1}{2}( \partial_1 - i\partial_2)$ and $\partial_{\overline{z}} = \tfrac{1}{2} (\partial_1 + i \partial_2)$. In the following we will always choose the square root of a complex number $w \in \mathbb{C} \setminus [0,+\infty)$ in such a way that $\textup{Im }\sqrt{w}  > 0$ applies. Eventually, when convenient $C$ denotes a generic constant whose value may change between different lines.

For a Hilbert space $\mathcal{H}$ we denote the associated inner product by $\langle \cdot, \cdot \rangle_{\mathcal{H}}$ and we use the convention that it is anti-linear in the first entry and linear in the second argument. If $\mathcal{H}$, $\mathcal{G}$ are Hilbert spaces and $A$ is a linear operator from $\mathcal{H}$ to $\mathcal{G}$, then its domain, range, and kernel are denoted by $\Dom(A)$, $\Ran(A)$, and $\Ker(A)$, respectively. For a closed operator $A$ in a Hilbert space $\mathcal{H}$ the resolvent set, spectrum,  and point spectrum  are  $\rho(A)$, $\sigma(A)$, and  $\sigma_\textup{p}(A)$. Moreover, if  $A$ is self-adjoint, then  we denote the discrete and essential spectrum by $\sigma_\textup{disc}(A)$ and $\sigma_\textup{ess}(A)$.

\subsection{Curves, function spaces, and trace theorems} \label{section_function_spaces}

Throughout this paper, we assume that $\Sigma$ is a compact non-self-intersecting \textit{bi-Lipschitz}  curve. By this, we mean that there exists an injective Lipschitz continuous arc-length parametrization $\zeta :  [0,L) \to \mathbb{R}^2$, where $L$ denotes the length of $\Sigma$, extended to $\mathbb{R}$ via periodicity, such  that $\zeta( [0, L)) =\Sigma$, $|\dot{\zeta}| =1 $ a.e. on $[0, L)$, and there exists a constant $C_\zeta>0$ such that 
	\begin{equation}\label{eq_bi_Lipschitz}
		C_\zeta |t-s | \leq   |\zeta(t) - \zeta(s)|  \leq |t-s|\qquad \forall t\in\R,\, \forall s \in   \overline{B(t,L/2)}.
	\end{equation}

\begin{remark} \label{rem_bi_Lipschitz}
If the extended function $\zeta$ is $C^2$-smooth, then \eqref{eq_bi_Lipschitz} is always true. The second inequality is clear even for Lipschitz continuous functions. To show the first inequality, define the function 
\begin{equation*}
r(s,t):=\begin{cases} \frac{|\zeta(t)-\zeta(s)|}{|t-s|}, & t\neq s,\\
		1, & t=s.
		\end{cases}
\end{equation*}
Employing the Taylor expansion, one checks that $r$ is continuous on $\R^2$. Moreover, $r$ is positive  on $\mathscr{M}:=\{(s,t)|\, t\in\R,\, s\in \overline{B(t,L/2)}\}$, due to the injectivity of $\zeta$ on $[0,L)$. Put $\mathscr{M}_L:=\{(s,t)|\, t\in[0,L],\, s\in \overline{B(t,L/2)}\}$. Taking into account also the periodicity of $\zeta$ and the compactness of $\mathscr{M}_L$, we get that ~\eqref{eq_bi_Lipschitz} is true with
$C_\zeta=\inf_{\mathscr{M}}r=\inf_{\mathscr{M}_L}r=\min_{\mathscr{M}_L}r>0$. 
\end{remark}

To shorten notation we use the expression  $\langle\cdot,\cdot\rangle_M$ instead of $\langle\cdot,\cdot\rangle_{L^2(M)}$ for the usual scalar product on $L^2(M)$, which is anti-linear in the first and linear in the second argument, where $M$ is either an open set in $\R^2$ or $\Sigma$; in the latter case the $L^2$ space is equipped with the standard arc-length measure $\sigma$.   The symbol $H^s(M)$ denotes the classical  Sobolev space of order $s \in \R$, where $M$ is an open subset of $\R^2$ or a sufficiently smooth curve $\Sigma$; cf. \cite[Chap.~3]{M00}. For $M= \Sigma$ one has then that
$H^s(\Sigma)$ is the dual space of $H^{-s}(\Sigma)$, which will be used later.

In the following, we introduce some further function spaces that will be used in several places in the course of the paper and on which extensions of the Dirichlet trace operator can be constructed. Define for $s \geq 0$ the Sobolev space
\begin{equation*}
H^{s}_\Delta (\Omega_\pm) := \left\{ f \in H^s(\Omega_\pm) \: \big| \: \Delta f \in L^2(\Omega_\pm) \right\}
\end{equation*}
and equip it with the norm $\|f\|^2_{H^{s}_\Delta (\Omega_\pm)} := \| f\|_{H^s(\Omega_\pm)}^2 + \| \Delta f\|_{L^2(\Omega_\pm)}^2$. Then,  $H^{s}_\Delta (\Omega_\pm)$ is a Hilbert space and $C^\infty_0(\overline{\Omega_\pm})$, which is the set of all functions mapping from $\Omega_\pm$ to $\C$ that have compactly supported $C^\infty$ extensions to $\R^2$, is dense in $H^s_\Delta(\Omega_\pm)$. This follows for $\Omega_+$ from \cite[Lem.~2.13]{BGM23} and can be proven in the same way for $\Omega_-$.  We also define the function spaces $H_{\dzbar}^{s}(\Omega_\pm)$ and  $H_{\dz}^{s}(\Omega_\pm)$ for $s \geq 0$ by
\begin{equation} \label{eq_Hdz_bar}
H_{\dzbar}^{s}(\Omega_\pm) := \left\{ f \in H^{s}(\Omega_\pm) \: \big| \: \dzbar f \in L^2(\Omega_\pm) \right\} 
\end{equation}
and
\begin{equation} \label{eq_Hdz}
H_{\dz}^{s}(\Omega_\pm) := \left\{ f \in H^{s}(\Omega_\pm) \: \big| \: \dz f \in L^2(\Omega_\pm) \right\},
\end{equation}
endowed with the norms 
\begin{equation*}
  \| f \|_{H_{\dz}^s(\Omega_\pm)}^2 := \| f\|_{H^s(\Omega_\pm)}^2 + \| \dz f\|_{L^2(\Omega_\pm)}^2
\end{equation*}
and 
\begin{equation*}
  \| f \|_{H_{\dzbar}^s(\Omega_\pm)}^2 := \| f\|_{H^s(\Omega_\pm)}^2 + \| \dzbar f\|_{L^2(\Omega_\pm)}^2.
\end{equation*}
Furthermore, we identify
\begin{equation}\label{def_H_s_Delta_Sigma}
H^s_\Delta(\mathbb{R}^2\setminus \Sigma)=H^s_\Delta(\Omega_+)\oplus H^s_\Delta(\Omega_-).
\end{equation}
Again, it can be shown as in \cite[Lem.~2.1]{BFSB17} that $H_{\dz}^s(\Omega_{\pm})$ and $H_{\dzbar}^s(\Omega_{\pm})$  are Hilbert spaces and that  the set $C_{0}^{\infty}(\overline{\Omega_{\pm}})$ is dense in $H_{\dz}^s(\Omega_{\pm})$ and $H_{\dzbar}^s(\Omega_{\pm})$. This observation allows to extend the continuous trace operator $\gamma_D^{\pm} : H^1(\Omega_{\pm}) \to H^{1/2}(\Sigma)$.

\begin{lemma} \label{Trace_Extension}
Let $s \in [\frac{1}{2},\frac{3}{2}]$. Then the mapping $C_0^{\infty}(\overline{\Omega_{\pm}}) \ni f \mapsto f |_{\Sigma}$ can be extended to linear and bounded operators 
\begin{equation*} 
\gamma_D^{\pm} : H_{\dzbar}^{s}(\Omega_{\pm}) \to H^{s-1/2}(\Sigma) \quad \text{and} \quad \gamma_D^{\pm} : H_{\dz}^{s}(\Omega_{\pm}) \to H^{s-1/2}(\Sigma).
\end{equation*}
Furthermore, for all $f \in  H_{\dzbar}^{s}(\Omega_{\pm})$ and all $g \in H_{\dz}^{s}(\Omega_{\pm})$ there holds 
\begin{equation} \label{GreenH12}
\langle \dzbar f , g \rangle_{\Omega_{\pm}} = - \langle f , \dz g \rangle_{\Omega_{\pm}} \pm \frac{1}{2} \langle  \bn \gamma_D^{\pm} f , \gamma_D^{\pm} g \rangle_{\Sigma}.
\end{equation}
If $\Sigma$ is a $C^{k}$-boundary with $k \geq 2$ then the above mapping properties remain valid for all $s \in [\tfrac{3}{2}, k]$. 
\end{lemma}

\begin{proof}
We show that the trace operator can be extended to a bounded map $\gamma_D^{\pm} : H_{\dzbar}^{s}(\Omega_{\pm}) \to H^{s-1/2}(\Sigma)$, the claim for $H_{\dz}^{s}(\Omega_{\pm})$ can be shown in the same way. For this, we use similar arguments as in the proof of \cite[Lem.~4.1]{BHSLS23}. Let $s \in [\frac{1}{2},\frac{3}{2}]$ be fixed. Define the space
\begin{equation*}
H^{s,-1}_\Delta (\Omega_{\pm}) := \left\{ f \in H^{s}(\Omega_{\pm}) \: \big| \: \Delta f \in H^{-1}(\Omega_{\pm}) \right\}
\end{equation*}
and equip it with the norm $\| f \|_{H^{s,-1}_\Delta (\Omega_{\pm})}^2 := \| f \|_{H^{s}(\Omega_{\pm})}^2 + \| \Delta f \|_{H^{-1}(\Omega_{\pm})}^2$. Then it follows from \cite[Thm.~3.6]{BGM23} that there exists a continuous extension $\gamma_D^{\pm} : H^{s,-1}_\Delta (\Omega_{\pm}) \to H^{s-1/2}(\Sigma)$ of the trace operator.  Due to $4 \dz \dzbar = \Delta$ and the continuity of the operator $\dz : L^2(\Omega_{\pm}) \to H^{-1}(\Omega_{\pm})$, the space $H^{s}_{\dzbar}(\Omega_{\pm})$ is continuously embedded in $H^{s,-1}_{\Delta}(\Omega_{\pm})$. Thus, the Dirichlet trace admits the claimed continuous extension defined on $H_{\dzbar}^{s}(\Omega_{\pm})$. The Green's identity \eqref{GreenH12} follows directly from the density of $C_0^{\infty}(\overline{\Omega_{\pm}})$ in $H^{s}_{\dzbar}(\Omega_{\pm})$ and $H^{s}_{\dz}(\Omega_{\pm})$, the continuity of the extended Dirichlet trace operator, and the classical integration by parts formula.
Finally, the claim for $C^k$-boundaries follows from the classical trace theorem, see, e.g., \cite[Thm.~3.37]{M00}.
\end{proof}

We will also need a suitable realization of the Neumann trace.  By \cite[Lem.~4.3]{M00} there exists a bounded operator $\gamma_N^{\pm} : H^1_{\Delta}(\Omega_{\pm}) \to H^{-1/2}(\Sigma)$ such that for all $f \in H^1_{\Delta}(\Omega_{\pm})$ and all $g \in H^1(\Omega_{\pm})$,
\begin{equation*} 
\langle \nabla f  , \nabla g \rangle_{\Omega_{\pm}} + \langle \Delta f , g \rangle_{\Omega_{\pm}} = \pm \big\langle \gamma_N^{\pm} f, \gamma_D^{\pm} g \big\rangle_{H^{-1/2}(\Sigma) \times H^{1/2}(\Sigma)}
\end{equation*}
is satisfied, where $\langle \cdot, \cdot \rangle_{H^{-1/2}(\Sigma) \times H^{1/2}(\Sigma)}$ denotes the sesquilinear duality product in $H^{-1/2}(\Sigma) \times H^{1/2}(\Sigma)$.  Note that by \cite[Thm.~5.4]{BGM23} one has for all $f \in H^{3/2}_\Delta(\Omega_{\pm})$ the relation $\gamma_N^{\pm} f = \nu \cdot \gamma_D^{\pm} \nabla f$ and $\gamma_N^\pm: H^{3/2}_\Delta(\Omega_\pm) \rightarrow L^2(\Sigma)$ is continuous. With an interpolation argument, see, e.g., \cite[Thms.~5.1\&14.3]{LM72}, it follows that $\gamma_N^{\pm} : H^s_{\Delta}(\Omega_{\pm}) \to H^{s-3/2}(\Sigma)$ is continuous for all $s \in [1,\frac{3}{2}]$.

\subsection{Distributions}\label{section_distributions}

The space of all test functions in an open set $M \subset \mathbb{R}^2$ is denoted by $\mathcal{D}(M)$, the symbol $\mathcal{D}'(M)$ is used for the distributions in $M$.
The pairing between the distributions and the test functions will be denoted by $(\cdot,\cdot)$. Note that this bracket is linear with respect to both arguments.

The symbol $\delta_\Sigma$ will stand for the single layer distribution supported on the hypersurface $\Sigma$. Furthermore, for $g \in L^2(\Sigma)$ we identify the distributions $g \delta_\Sigma$, $\partial_z (g\delta_\Sigma)$, and $\partial_{\overline{z}} (g \delta_\Sigma)$ with the inner product on $L^2(\Sigma)$ in the following manner:
\begin{equation*}
\left.  \begin{split}
    (g \delta_\Sigma,\varphi) &:= \langle \overline{g}, \varphi \upharpoonright \Sigma \rangle_\Sigma, \\
    (\partial_z(g \delta_\Sigma),\varphi) &:= - \langle \overline{g},\partial_z \varphi \upharpoonright \Sigma\rangle_\Sigma,  \\
    (\partial_{\overline{z}}(g \delta_\Sigma),\varphi) &:= - \langle \overline{g},\partial_{\overline{z}} \varphi \upharpoonright \Sigma\rangle_\Sigma.
  \end{split}\, \right\}\quad\varphi \in \mathcal{D}(\R^2)
\end{equation*}
We will do the same for $\mathbb{C}^2$-valued test functions, i.e., for $G\in L^2(\Sigma;\mathbb{C}^{2 \times 2})$ the distribution $G\delta_\Sigma$ will be defined as 
\begin{equation*}
  (G\delta_\Sigma,\varphi):=\int_\Sigma G (\varphi \upharpoonright \Sigma) \dd \sigma, \quad  \varphi\in\mathcal{D}(\mathbb{R}^2;\mathbb{C}^2) .
\end{equation*}
Note that the distributions above can be extended, similarly as the $\delta$-distribution in \cite{Kur96}, to every  (possibly $\mathbb{C}^2$-valued) function $\varphi$ such that $\varphi_\pm := \varphi \upharpoonright \Omega_\pm$ admit well-defined square integrable Dirichlet traces:

\begin{equation*}
\begin{split}
(g\delta_\Sigma,\varphi)&:= \int_\Sigma g\frac{\gamma_D^+\varphi_+ +\gamma_D^-\varphi_-}{2} \dd \sigma,\\
(\dz(g\delta_\Sigma),\varphi)&:= -\int_\Sigma g\frac{\gamma_D^+(\dz\varphi_+)+\gamma_D^-(\dz\varphi_-)}{2} \dd \sigma,\\
(\dzbar(g\delta_\Sigma),\varphi)&:= -\int_\Sigma g\frac{\gamma_D^+(\dzbar\varphi_+)+\gamma_D^-(\dzbar\varphi_-)}{2}\dd \sigma,
\end{split}
\end{equation*}
and
$$(G\delta_\Sigma,\varphi):=\int_\Sigma G \frac{\gamma_D^+\varphi_+ +\gamma_D^-\varphi_-}{2} \dd \sigma.$$

Moreover, we will extend the usual physical \textit{bra-ket} notation to certain distributions. Let $f,g\in L^2(\Sigma)$ and $F,G\in L^2(\Sigma;\mathbb{C}^{2 \times 2})$ then for every $\varphi$ as above we set

\begin{align*}
|f\delta_\Sigma\rangle\langle g\delta_\Sigma|\varphi &:= (\overline{g}\delta_\Sigma, \varphi)f\delta_\Sigma,\\
|f\delta_\Sigma\rangle\langle\dz (g\delta_\Sigma)|\varphi &:= (\dzbar(\overline{g}\delta_\Sigma), \varphi)f\delta_\Sigma,\\
|\dz(f\delta_\Sigma)\rangle\langle g\delta_\Sigma|\varphi &:= (\overline{g}\delta_\Sigma, \varphi)\dz(f\delta_\Sigma),\\
|\dz(f\delta_\Sigma)\rangle\langle \dz (g\delta_\Sigma)|\varphi &:= (\dzbar(\overline{g}\delta_\Sigma), \varphi)\dz(f\delta_\Sigma), \\
|F\delta_\Sigma\rangle\langle G\delta_\Sigma|\varphi &:= F(G^\ast\delta_\Sigma,\varphi)\delta_\Sigma.
\end{align*}

\subsection{Integral operators}\label{section_int.oper.}

	In this section we collect several useful results about integral operators and their mapping properties. We start with two abstract results which are basically consequences of the Schur test (see \cite[Ex.~III~2.4]{kato} or \cite[Thm.~6.24]{W80}) and that will be used frequently, in particular in the Sections~\ref{section_spectrum}\&\ref{section_nrl}. First, we discuss the boundedness of certain boundary integral operators. Recall that $\Sigma$ is always a compact non-self-intersecting bi-Lipschitz smooth curve of length $L$ as described at the beginning of Section~\ref{section_function_spaces}.

	\begin{proposition}\label{prop.-sing.int.op.on Sigma}
		Let $a:\mathbb{R}^2\to \mathbb{C}$ be a measurable function, $L$ be the length of $\Sigma$,  $C_\zeta >0$ be the constant from~\eqref{eq_bi_Lipschitz}, and  $\alpha: (0,+\infty) \to (0,+\infty)$   be a monotonically decreasing   function  such that $\alpha$ is integrable on $(0, \frac{C_\zeta L}{2})$. Moreover, assume that for almost all $x \in \R^2\setminus\{0\}$ with $|x | \leq  \tfrac{L}{2}$
		\begin{equation}\label{eq.-sing.int.op.on Sigma}
			|a(x)| \leq \alpha(|x|) 
		\end{equation}
		is fulfilled.
		Then the linear operator $A: L^2(\Sigma) \to L^2(\Sigma)$ defined by 
		\begin{equation*}
			\begin{aligned}
				(A \varphi )(x_\Sigma) &= \int_{\Sigma}a(x_\Sigma-y_\Sigma) \varphi(y_\Sigma) \dd \sigma(y_\Sigma), \quad &\varphi \in L^2 (\Sigma),\, x_\Sigma \in \Sigma,
			\end{aligned}
		\end{equation*}
		is well-defined and for the operator norm the  estimate
		\begin{equation*}
			\begin{aligned}
				\| A\| & \leq 2 C_\zeta^{-1} \int_0^{C_\zeta L/2} \alpha(s) \dd s
			\end{aligned}
		\end{equation*}
		is valid.
	\end{proposition}
	\begin{proof}
		We start by  introducing the operator
		\begin{equation} \label{def_iota_zeta}
			\iota_{\zeta}: L^2(\Sigma) \to L^2((0,L)), \qquad \varphi \mapsto \varphi \circ \zeta.
		\end{equation}
		Since  $|\dot{\zeta}| = 1$ a.e., $\iota_\zeta$ is  an isometric isomorphism. Thus, $A$ is well-defined if and only if $\widetilde{A} := \iota_{\zeta}A\iota_{\zeta}^{-1}:  L^2((0,L)) \to L^2((0,L))$ is well-defined. A transformation of coordinates shows that $\widetilde{A}$ acts as
		\begin{equation*}
			(\widetilde{A} f)(t) = \int_0^L a(\zeta(t) - \zeta(s)) f(s) \dd s, \qquad f \in L^2((0,L)),~t \in (0,L).
		\end{equation*}
		Define the number $K \in [0, +\infty]$ by
		\begin{equation} \label{eq_Schur_test_1}
			K := \bigg(\sup_{t \in (0,L) } \int_0^L \alpha(|\zeta(t) - \zeta(s)|)\, \dd s\bigg)^{1/2} \bigg(\sup_{s \in (0,L)}  \int_0^L \alpha(|\zeta(t) - \zeta(s)|)\, \dd t\bigg)^{1/2}.
		\end{equation}
		Then, by \eqref{eq.-sing.int.op.on Sigma} and the Schur test, see, e.g., \cite[Ex.~III~2.4]{kato} or \cite[Thm.~6.24]{W80},  $\widetilde{A}$ is well-defined if $K < +\infty$, and in this case $\| A \| = \| \widetilde{A} \| \leq K$.
		Thus, it remains to estimate the constant $K$ in~\eqref{eq_Schur_test_1}. For $t \in (0,L)$
		the periodicity of $\zeta$, \eqref{eq_bi_Lipschitz}, and the monotonicity of $\alpha$ yield
		\begin{equation*}
			\begin{aligned}
				\int_0^L \alpha(|\zeta(t) - \zeta(s)|) \dd s &= \int_{t-L/2}^{t+L/2}  \alpha(|\zeta(t) - \zeta(s)|) \dd s\\
				& \leq \int_{t-L/2}^{t+L/2} \alpha(C_\zeta |t-s|)  \dd s\\
				& = 2C_\zeta^{-1} \int_0^{C_\zeta L/2} \alpha(s) \dd s.
			\end{aligned}
		\end{equation*}
		Since the latter estimate is independent of $t$, we conclude that
		\begin{equation*}
			\bigg(\sup_{t \in (0,L)}\int_0^L |\alpha(\zeta(t) - \zeta(s))| \dd s\bigg)^{1/2} \leq \bigg(2C_\zeta^{-1} \int_0^{C_\zeta L/2} \alpha(s) \dd s\bigg)^{1/2}.
		\end{equation*}
		By symmetry, the second term in \eqref{eq_Schur_test_1} can be estimated in the same way. Hence, we conclude that the constant $K$ in~\eqref{eq_Schur_test_1} can be estimated by
		\begin{equation*}
            K \leq 2 C_\zeta^{-1} \int_0^{C_\zeta L/2} \alpha(s) \dd s,
		\end{equation*}
		which finishes the proof of this proposition.
	\end{proof}

    The next proposition is a counterpart of Proposition~\ref{prop.-sing.int.op.on Sigma} for potential operators.

	\begin{proposition}\label{prop.-pot.int.op.}
		Let $a:\mathbb{R}^2\to \mathbb{C}$ be measurable, $L$ be the length of the curve $\Sigma$,  $C_\zeta >0$ be the constant from \eqref{eq_bi_Lipschitz}, $\theta \in [0,1)$, and  $\alpha: (0,+\infty) \to (0,+\infty)$   be a monotonically decreasing   function  such that $\alpha^{2\theta}$ and $ (0,+\infty) \ni r \mapsto (\alpha(r))^{2(1-\theta)}r$ are integrable on $(0,C_\zeta \frac{L}{4})$ and $(0,+\infty)$, respectively. Moreover, assume that for almost all $x \in \R^2\setminus\{0\}$
		\begin{equation}\label{eq.-pot.int.op.}
			|a(x)| \leq \alpha(|x|)
		\end{equation}
		is fulfilled.
		Then the linear operator $A:L^2(\Sigma)\to L^2(\mathbb{R}^2)$ defined by 
		\begin{equation*}
			\begin{aligned}
				(A \varphi )(x) &= \int_{\Sigma}a(x-y_\Sigma) \varphi(y_\Sigma) \dd \sigma(y_\Sigma), \quad &\varphi \in L^2 (\Sigma), x \in \mathbb{R}^2,
			\end{aligned}
		\end{equation*}
		is well-defined and for the operator norm the  estimate
		\begin{equation*}
			\begin{aligned}
				\| A \| &\leq  \bigg(\frac{4}{C_\zeta}\int_0^{C_\zeta L/4} (\alpha(s))^{2\theta} \dd s\bigg)^{1/2} \bigg( 2\pi \int_0^{\infty} (\alpha(r))^{2(1-\theta)}r \dd r \bigg)^{1/2}
			\end{aligned}
		\end{equation*}
		is valid.
	\end{proposition}
	\begin{proof}
        As in the proof of Proposition~\ref{prop.-sing.int.op.on Sigma} we make use of the isometric isomorphism $\iota_\zeta$ defined by~\eqref{def_iota_zeta} and consider the operator $\widetilde{A} := A \iota_\zeta^{-1}: L^2((0, L)) \rightarrow L^2(\mathbb{R}^2)$ which acts as
		\begin{equation*}
			(\widetilde{A} f)(x) = \int_0^L a(x - \zeta(s)) f(s) \dd s, \qquad f \in L^2((0, L)),~ x \in \R^2.
		\end{equation*}
		Defining $K \in [0, +\infty]$ by
		\begin{equation} \label{eq_Schur_test_2}
			K:= \bigg(\sup_{x \in \R^2} \int_0^L (\alpha(|x - \zeta(s)|))^{2\theta}\dd s\bigg)^{1/2} \bigg(\sup_{s \in (0,L)}  \int_{\mathbb{R}^2} (\alpha(|x - \zeta(s)|))^{2(1-\theta)} \dd x\bigg)^{1/2},
		\end{equation}
		we get with \eqref{eq.-pot.int.op.} and the Schur test again that $\widetilde{A}$ is well-defined, if $K < +\infty$, and in this case $\| A \| = \| \widetilde{A} \| \leq K$.
		We start by estimating the first term in \eqref{eq_Schur_test_2}. Fix $x \in \R^2$. Choose $t \in (0,L]$ such that $\min_{s\in\R}|x-\zeta(s)| = |x-\zeta(t)|$. Then,
		\begin{equation*}
			|\zeta(t) - \zeta(s) |  \leq |\zeta(t)-x| + |x-\zeta(s)| \leq 2 |x-\zeta(s)| \quad \forall s \in \R. 
		\end{equation*}
		Combining this estimate with \eqref{eq_bi_Lipschitz} gives us
		\begin{equation*}
			\frac{C_\zeta}{2}  |t-s |  \leq |x-\zeta(s)|  \quad \forall s \in \overline{B(t,L/2)}. 
		\end{equation*}
		Thus, we can estimate the first term in \eqref{eq_Schur_test_2} in the same way as in the proof of Proposition~\ref{prop.-sing.int.op.on Sigma} by
		\begin{equation*}
			\bigg(\frac{4}{C_\zeta}\int_0^{C_\zeta L/4} (\alpha(s))^{2\theta} \dd s\bigg)^{1/2},
		\end{equation*}
		which is finite by assumption.
		Furthermore, using polar coordinates one gets for $s \in (0,L)$ that
		\begin{equation*}
			\int_{\mathbb{R}^2} (\alpha(|x - \zeta(s)|))^{2(1-\theta)} \dd x = 2 \pi \int_0^\infty (\alpha(r))^{2(1-\theta)} r  \dd r,
		\end{equation*} 
		and therefore the second term in \eqref{eq_Schur_test_2} can be estimated by the square root of $2 \pi \int_0^\infty (\alpha(r))^{2(1-\theta)} r \, \dd r$. This finishes the proof of the proposition.
	\end{proof}

	Throughout this paper, integral operators with kernels, where the modified Bessel functions $K_0$ and $K_1$ of the second kind appear, will play an important role. Hence, we summarize some of their properties:
	
	\begin{lemma} \label{lemma_Bessel_functions}
		The following statements are valid for the modified Bessel functions $K_0$ and $K_1$.
		\begin{itemize}
			\item[\textup{(i)}] For all $t > 0$ one has $K_0(t)>0$ and $K_1(t) > 0$.
			\item[\textup{(ii)}] $K_0'(t) = -K_1(t)$ and $t K_1'(t) = - t K_0(t) - K_1(t)$ for all $t \in \mathbb{C} \setminus (- \infty,0]$.
			\item[\textup{(iii)}] There exist analytic functions $g_1, g_2 : \mathbb{C} \to \mathbb{C}$ such that for all $t \in \mathbb{C} \setminus (- \infty,0]$
			\begin{equation*} 
				K_1(t) = \frac{1}{t} + t  g_1(t^2) \ln(t) + t g_2(t^2).
			\end{equation*}
			\item[\textup{(iv)}]   There exist constants $\kappa_1,\,\tilde\kappa_1 > 0$ such that for all $t \in \mathbb{C} \setminus (- \infty,0]$
			\begin{equation*}
				|K_0(t)| \leq \kappa_1 \frac{ 1+ |\ln|t||  }{1+ \sqrt{|t|}} \ee^{-\textup{Re}(t)}\leq \tilde\kappa_1 |t|^{-1/4}\ee^{-\textup{Re}(t)}.
			\end{equation*}
			\item[\textup{(v)}]  There exists a constant $\kappa_2 > 0$ such that for all $t \in \mathbb{C} \setminus(- \infty,0]$ and $j \in \{0,1\}$
			\begin{equation*}
				\left| \frac{d^j}{d t^j} K_1(t) \right| \leq \kappa_2 \left(\frac{1}{\sqrt{|t|}} + \frac{1-j}{|t|} + \frac{j}{|t|^2} \right) \ee^{-\textup{Re}(t)}.
			\end{equation*}
		\end{itemize}
	\end{lemma}
	\begin{proof}
		Item~(i) can be deduced from the integral representations of $K_0$ and $K_1$ in \cite[\S~10.32]{DLMF} and item~(ii) follows from\cite[\S~10.29(i)]{DLMF}. For the proof of item~(iii) one merely has to use the power series expansion of $K_1$ in \cite[\S~10.31]{DLMF}, cf. the proof of \cite[Lem.~3.2]{BHOP20} for details. The first claimed inequality in~(iv) follows from the asymptotic expansion of $K_0$, see, e.g., \cite[\S~10.25(ii) and \S~10.30(i)]{DLMF}, the second inequality is then obtained immediately from the first one. Similarly, the asymptotic expansion of $K_1$ from \cite[\S~10.25(ii) and \S~10.30(i)]{DLMF}  leads to the inequality in item~(v) for $j=0$. For the case $j=1$ one obtains the claimed inequality from the recurrence relation of the modified Bessel functions in  item~(ii)  and the asymptotics of $K_0$ and $K_1$ for small and large arguments. 
	\end{proof}

In the remaining part of this subsection we introduce and discuss several families of integral operators that will be used throughout the paper. We define for $w \in \mathbb{C} \setminus [0, +\infty)$ the single layer potential $SL(w)$ acting on $\varphi \in L^2(\Sigma)$ by 
\begin{equation} \label{def_single_layer_potential}
  SL(w) \varphi(x) = \frac{1}{2 \pi} \int_\Sigma  K_0(-i\sqrt{w} |x-y_\Sigma|) \varphi(y_\Sigma) \dd \sigma(y_\Sigma), \quad x \in \mathbb{R}^2 \setminus \Sigma.
\end{equation}
Taking the second bound from Lemma~\ref{lemma_Bessel_functions} (iv) and Proposition~\ref{prop.-pot.int.op.} with $\theta=1/2$ into account, one finds that $SL(w)$ is bounded from $L^2(\Sigma)$ to $L^2(\mathbb{R}^2)$.  
In the next lemma we collect the properties of the single layer potential needed for the further considerations in this paper.
	\begin{lemma} \label{Properties_SL}
		Let $w \in \mathbb{C} \setminus [0,+\infty)$  and $SL(w)$ be defined as in \eqref{def_single_layer_potential}. Then the following statements hold.
		\begin{itemize}
			\item[\textup{(i)}] $SL(w): L^2(\Sigma) \rightarrow H^{3/2}_\Delta(\mathbb{R}^2 \setminus \Sigma)$ is bounded and $\Ran (SL(w)) \subset H^1(\mathbb{R}^2)$. Moreover, if 
    $\Sigma$ is a $C^{k}$-boundary with $k \geq 2$, then 
	\begin{equation} \label{single_layer_potential_mapping_properties}
		SL(w): H^{s-1/2}(\Sigma) \rightarrow H^{s+1}_\Delta(\mathbb{R}^2 \setminus \Sigma), \quad s \in \left[\frac{1}{2},k-1\right],
\end{equation}
is well-defined and bounded.
			\item[\textup{(ii)}] $(-\Delta - w) SL(w) \varphi = 0$ in $\mathbb{R}^2 \setminus \Sigma$ for all $\varphi \in L^2(\Sigma)$.
        \item[\textup{(iii)}]
        For every $\varphi \in L^2(\Sigma)$ there holds  
			\begin{equation*}
				\gamma_D^{+}\big( SL(w) \varphi \big)_{+} - \gamma_D^{-} \big( SL(w) \varphi \big)_{-} = 0, \:  \gamma_N^{+}\big( SL(w) \varphi \big)_{+} - \gamma_N^{-} \big( SL(w) \varphi \big)_{-} = \varphi.
			\end{equation*}
		\end{itemize}
	\end{lemma}
	\begin{proof}
		If $\Sigma$ is a Lipschitz boundary, then the asserted claims follow immediately from \cite[Eq.~ (3.17) and Lem.~3.3~(i)]{HU20}, \cite[Thm.~6.11]{M00}, and \cite[Thm.~7.4]{MT00}. Thus it remains to show the improved mapping properties from item~(i) if $\Sigma$ is a $C^{k}$-boundary and if $s \in [\tfrac{1}{2}, k-1]$.  For this purpose, let $\chi \in C_0^{\infty}(\R^2)$ be a suitable cut-off function with $\chi \equiv 1$ on an open neighbourhood of $ \overline{\Omega_{+}}$.  Then it follows from \cite[Cor.~6.14]{M00} that $(\chi SL(w) \varphi)_\pm \in H^{s+1}(\Omega_{\pm})$ is valid for all $\varphi \in H^{s-1/2}(\Sigma)$.  Furthermore, it follows from the large argument asymptotic expansions of the derivatives of the modified Bessel function $K_0$ from \cite[\S~10.25(ii) and \S~10.29(i)]{DLMF} and the integral representation of the single layer potential in \eqref{def_single_layer_potential} that $(1-\chi) SL(w) \varphi \in H^k(\Omega_{\pm})$. Hence, the mapping in~\eqref{single_layer_potential_mapping_properties} is well-defined. Since $SL(w): L^2(\Sigma) \rightarrow H^{3/2}_\Delta(\mathbb{R}^2 \setminus \Sigma)$ is bounded, it follows with the closed graph theorem, that also the mapping in~\eqref{single_layer_potential_mapping_properties} is bounded.
	\end{proof}

We will also make use of the integral operators $WL(w)$ and $\widetilde{WL}(w)$ that are formally defined via their action on $\varphi \in L^2(\Sigma)$ by
\begin{equation} \label{def_Psi}
  WL(w) \varphi(x) := \frac{i \sqrt{w}}{4 \pi} \int_\Sigma \frac{\overline{{\bf x}} - \overline{{\bf y}_\Sigma}}{|x-y_\Sigma|}  K_1(-i\sqrt{w} |x-y_\Sigma|) \varphi(y_\Sigma) \dd \sigma(y_\Sigma), \quad x \in \mathbb{R}^2 \setminus \Sigma,
\end{equation}
and 
\begin{equation} \label{def_Xi}
  \widetilde{WL}(w) \varphi(x) := \frac{i \sqrt{w}}{4 \pi} \int_\Sigma \frac{{\bf x} - {\bf y}_\Sigma}{|x-y_\Sigma|}  K_1(-i\sqrt{w} |x-y_\Sigma|) \varphi(y_\Sigma) \dd \sigma(y_\Sigma), \quad x \in \mathbb{R}^2 \setminus \Sigma.
\end{equation}
While it follows from Lemma~\ref{lemma_Bessel_functions}~(v) and Proposition~\ref{prop.-pot.int.op.} (for $\theta = 1/4$) that $WL(w)$ and $\widetilde{WL}(w)$ are bounded between $L^2$-spaces, one can even show better mapping properties. 

\begin{proposition} \label{proposition_Psi}
  Let $w \in \mathbb{C} \setminus [0, +\infty)$ and $WL(w), \widetilde{WL}(w)$ be defined by~\eqref{def_Psi} and~\eqref{def_Xi}, respectively. Then, the following holds:
  \begin{itemize}
    \item[\textup{(i)}] $WL(w) = \partial_z SL(w)$. In particular, $WL(w): L^2(\Sigma) \rightarrow H^{1/2}_\Delta(\mathbb{R}^2 \setminus \Sigma)$ and $\partial_{\overline{z}} WL(w) = \frac{1}{4} \Delta SL(w) = -\frac{w}{4} SL(w): L^2(\Sigma) \rightarrow H^{3/2}_\Delta(\mathbb{R}^2 \setminus \Sigma)$ are bounded and $\Ran (\partial_{\overline{z}} WL(w)) \subset H^1(\mathbb{R}^2)$. Moreover, if 
    $\Sigma$ is a $C^{k}$-boundary with $k \geq 2$, then 
	\begin{equation} \label{WL_C_k}
		WL(w) : H^{s-1/2}(\Sigma) \to H^{s}_{\dzbar}(\R^2 \setminus \Sigma), \quad s \in \left[\frac{1}{2},k-1\right],
\end{equation}
is well-defined and bounded.
    \item[\textup{(ii)}] $\widetilde{WL}(w) = \partial_{\overline{z}} SL(w)$. In particular, $\widetilde{WL}(w): L^2(\Sigma) \rightarrow H^{1/2}_\Delta(\mathbb{R}^2 \setminus \Sigma)$ and $\partial_{z} \widetilde{WL}(w) = \frac{1}{4} \Delta SL(w) = -\frac{w}{4} SL(w): L^2(\Sigma) \rightarrow H^{3/2}_\Delta(\mathbb{R}^2 \setminus \Sigma)$ are bounded and $\Ran (\partial_{z} \widetilde{WL}(w)) \subset H^1(\mathbb{R}^2)$. Moreover, if 
    $\Sigma$ is a $C^{k}$-boundary with $k \geq 2$, then 
\begin{equation} \label{WL_tilde_C_k}
		\widetilde{WL}(w) : H^{s-1/2}(\Sigma) \to H^{s}_{\dz}(\R^2 \setminus \Sigma), \quad s \in \left[\frac{1}{2},k-1\right],
\end{equation}
is well-defined and bounded.
    \item[\textup{(iii)}] $(-\Delta - w) WL(w) \varphi = 0$ and $(-\Delta - w) \widetilde{WL}(w) \varphi = 0$ in $\mathbb{R}^2 \setminus \Sigma$ for all $\varphi \in L^2(\Sigma)$.
  \end{itemize}
\end{proposition}
\begin{proof}
The claims for $WL(w)$ can be shown in exactly the same way as \cite[Prop.~2.1]{BHS23} taking into account that $-2 i WL(w) = \Psi_w$, where $\Psi_w$ is the operator considered in \cite{BHS23}, we omit the details here. The claims for $\widetilde{WL}(w)$ follow from the fact that $\widetilde{WL}(w) \varphi = \overline{WL(\overline{w}) \overline{\varphi}}$ holds for all $\varphi \in L^2(\Sigma)$.
  The improved mapping properties for $C^{k}$-boundaries $\Sigma$ are a simple consequence of Lemma \ref{Properties_SL} (i), $WL(w) = \partial_z SL(w)$, and  $\widetilde{WL}(w) = \partial_{\overline{z}} SL(w)$.
\end{proof}

We will also make use of boundary integral operator versions of $SL(w)$, $WL(w)$, and $\widetilde{WL}(w)$. We define the single layer boundary integral operator 
\begin{equation} \label{def_single_layer_boundary_integral_operator}
  S(w) := \gamma_D SL(w): L^2(\Sigma) \rightarrow H^1(\Sigma)
\end{equation}
and note that this operator acts on $\varphi \in L^2(\Sigma)$ as 
\begin{equation} \label{IntegralRepresSLBIO} 
  S(w) \varphi(x_\Sigma) = \frac{1}{2 \pi} \int_\Sigma  K_0(-i\sqrt{w} |x_\Sigma-y_\Sigma|) \varphi(y_\Sigma) \dd \sigma(y_\Sigma), \quad x_\Sigma \in \Sigma.
\end{equation}
If $\Sigma$ is a $C^{k}$-boundary with $k \geq 2$, then the mapping properties of $S(w)$ improve. Using the mapping properties of the single layer potential in Lemma \ref{Properties_SL} (i) and Lemma~\ref{Trace_Extension} one gets that
	\begin{equation} \label{S_C_k}
		S(w) : H^{s-1/2}(\Sigma) \to H^{s+1/2}(\Sigma), \quad s \in [0,k-1],
	\end{equation}
is well-defined and bounded.

The operator $S(w)$ is also related to certain jump relations for $WL(w)$. The following lemma can be also shown in exactly the same way as \cite[Prop.~2.1]{BHS23} using the same arguments as for the previous proposition. Therefore, we will leave out its proof.

\begin{lemma} \label{lemma_jump_relations}
  Let $w \in \mathbb{C} \setminus [0, +\infty)$ and $WL(w), \widetilde{WL}(w)$ be defined by~\eqref{def_Psi} and~\eqref{def_Xi}, respectively. Then for any $\varphi \in L^2(\Sigma)$ the jump relations
    \begin{equation*}
      \begin{split}
        2 {\bf n} \big(\gamma_D^+ (WL(w) \varphi)_+ - \gamma_D^- (WL(w) \varphi)_- \big) &= \varphi, \\
        \frac{1}{2} \big(\gamma_D^+ (\partial_{\overline{z}} WL(w) \varphi)_+ + \gamma_D^- (\partial_{\overline{z}} WL(w) \varphi)_- \big) &= -\frac{w}{4} S(w) \varphi,
      \end{split}
    \end{equation*}
    and
    \begin{equation*}
      \begin{split}
        2 \overline{\bn} \big(\gamma_D^+ (\widetilde{WL}(w) \varphi)_+ - \gamma_D^- (\widetilde{WL}(w) \varphi)_- \big) &= \varphi, \\
        \frac{1}{2} \big(\gamma_D^+ (\partial_{z} \widetilde{WL}(w) \varphi)_+ + \gamma_D^- (\partial_{z} \widetilde{WL}(w) \varphi)_- \big) &= -\frac{w}{4} S(w) \varphi,
      \end{split}
    \end{equation*}
    hold.
\end{lemma}

Since it follows from Proposition~\ref{proposition_Psi} that $WL(w): L^2(\Sigma) \rightarrow H^{1/2}_{\dzbar}(\mathbb{R}^2 \setminus \Sigma)$ and $\widetilde{WL}(w): L^2(\Sigma) \rightarrow H^{1/2}_{\dz}(\mathbb{R}^2 \setminus \Sigma)$, we can define with Lemma~\ref{Trace_Extension} for $w \in \mathbb{C} \setminus [0, +\infty)$ the operators
\begin{equation} \label{def_L}
  W(w): L^2(\Sigma) \rightarrow L^2(\Sigma), \quad W(w) \varphi := \frac{1}{2} \big( \gamma_D^+ (WL(w) \varphi)_+ + \gamma_D^- (WL(w) \varphi)_- \big),
\end{equation}
and
\begin{equation} \label{def_tilde_L}
  \widetilde{W}(w): L^2(\Sigma) \rightarrow L^2(\Sigma), \quad \widetilde{W}(w) \varphi := \frac{1}{2} \big( \gamma_D^+ (\widetilde{WL}(w) \varphi)_+ + \gamma_D^- (\widetilde{WL}(w) \varphi)_- \big).
\end{equation}
If $\Sigma$ is a $C^{k}$-boundary with $k \geq 2$, then the mapping properties of $W(w)$ and $\widetilde{W}(w)$ improve. Indeed, one can show with the help of~Proposition~\ref{proposition_Psi}~(i) and the boundedness of the  trace operator $\gamma_D^{\pm} : H^s_{\dzbar}(\Omega_{\pm}) \to H^{s-1/2}(\Sigma)$ from Lemma \ref{Trace_Extension} that
	\begin{equation} \label{W_C_k}
		W(w) : H^{s-1/2}(\Sigma) \to H^{s-1/2}(\Sigma), \quad s \in [\tfrac{1}{2},k-1],
	\end{equation}
	is well-defined and bounded. Similarly, Proposition~\ref{proposition_Psi}~(ii) implies that
	\begin{equation} \label{W_tilde_C_k}
		\widetilde{W}(w) : H^{s-1/2}(\Sigma) \to H^{s-1/2}(\Sigma), \quad s \in [\tfrac{1}{2},k-1],
	\end{equation}
	is well-defined and bounded as well.

In the following lemma we will see that $W(w)$ and $\widetilde{W}(w)$ can be represented as strongly singular boundary integral operators and that $\widetilde{W}(w) = -W(\overline{w})^*$.

\begin{lemma} \label{lemma_W}
  Let $w \in \mathbb{C} \setminus [0, +\infty)$. Then, for all $\varphi \in L^2(\Sigma)$ and almost all $x_\Sigma \in \Sigma$ the representations 
  \begin{equation*}
    W(w) \varphi(x_\Sigma) = \frac{i \sqrt{w}}{4 \pi} \lim_{\varepsilon \rightarrow 0+} \int_{\Sigma \setminus B(x_\Sigma, \varepsilon)} \frac{\overline{{\bf x}_\Sigma} - \overline{{\bf y}_\Sigma}}{|x_\Sigma-y_\Sigma|}  K_1(-i\sqrt{w} |x_\Sigma-y_\Sigma|) \varphi(y_\Sigma) \dd \sigma(y_\Sigma)
  \end{equation*}
  and
  \begin{equation*}
    \widetilde{W}(w) \varphi(x_\Sigma) = \frac{i \sqrt{w}}{4 \pi} \lim_{\varepsilon \rightarrow 0+} \int_{\Sigma \setminus B(x_\Sigma, \varepsilon)} \frac{{\bf x}_\Sigma - {\bf y}_\Sigma}{|x_\Sigma-y_\Sigma|}  K_1(-i\sqrt{w} |x_\Sigma-y_\Sigma|) \varphi(y_\Sigma) \dd \sigma(y_\Sigma)
  \end{equation*}
  hold, where $B(x_\Sigma, \varepsilon)$ denotes the ball with center $x_\Sigma$ and radius $\varepsilon$. In particular, $\widetilde{W}(w) = -W(\overline{w})^*$.
\end{lemma}
\begin{proof}
  We prove the claimed integral representation of $W(w)$, the one of $\widetilde{W}(w)$ can be proved in the same way. Having these implies then also $\widetilde{W}(w) = -W(\overline{w})^*$ due to the explicit form of the integral kernels. 

  Let $\varphi \in L^2(\Sigma)$ and $x_\Sigma \in \Sigma$ be fixed. Recall that by Lemma~\ref{lemma_Bessel_functions}~(iii) there exist analytic functions $g_1$ and $g_2$ such that for all $t \in \mathbb{C} \setminus (- \infty,0]$
  \begin{equation} \label{ExpansionBesselK}
    K_1(t) = \frac{1}{t} + t  g_1(t^2) \ln(t) + t g_2(t^2).
  \end{equation}
  By dominated convergence we have
  \begin{equation} \label{limit_g}
    \begin{split}
      \lim_{\Omega_{\pm} \ni x \to x_\Sigma} &\frac{i \sqrt{w}}{4 \pi}  \int_{\Sigma}   \big(-i\sqrt{w} |x-y_\Sigma| \big) \Big( g_1(-w |x-y_\Sigma|^2) \ln (-i\sqrt{w} |x-y_\Sigma|) \\
      & \quad\qquad+ g_2(-w |x-y_\Sigma|^2) \Big)\frac{\overline{{\bf x}} - \overline{{\bf y}_\Sigma}}{|x-y_\Sigma|}  \varphi(y_\Sigma) \dd \sigma(y_\Sigma) \\
      &= \frac{i \sqrt{w}}{4 \pi}  \int_{\Sigma}   \big(-i\sqrt{w} |x_\Sigma-y_\Sigma| \big) \Big( g_1(-w |x_\Sigma-y_\Sigma|^2) \ln (-i\sqrt{w} |x_\Sigma-y_\Sigma|) \\
      & \quad\qquad+ g_2(-w |x_\Sigma-y_\Sigma|^2) \Big)\frac{\overline{{\bf x}_\Sigma} - \overline{{\bf y}_\Sigma}}{|x_\Sigma-y_\Sigma|}  \varphi(y_\Sigma) \dd \sigma(y_\Sigma).
    \end{split}
  \end{equation}
  Moreover, by \cite[p.~1071]{EFV92} we have for the non-tangential limit at $x_{\Sigma}$
  \begin{equation} \label{NTLimitsdzSL}
    \begin{split}
      -\lim_{\Omega_{\pm} \ni x \to x_\Sigma}& \frac{1}{4 \pi}  \int_{\Sigma}   \frac{\overline{{\bf x}} - \overline{{\bf y}_\Sigma}}{|x-y_\Sigma|^2}  \varphi(y_\Sigma) \dd \sigma(y_\Sigma) \\
      &= \pm \frac{1}{4} \overline{{\bf n}(x_\Sigma)} \varphi(x_\Sigma) - \frac{1}{4 \pi} \lim_{\varepsilon \to 0^{+}} \int_{\Sigma \setminus B(x_\Sigma,\varepsilon)}  \frac{\overline{{\bf x}_\Sigma} - \overline{{\bf y}_\Sigma}}{|x_\Sigma-y_\Sigma|^2}  \varphi(y_\Sigma) \dd \sigma(y_\Sigma).
    \end{split}
  \end{equation}
  By combining~\eqref{limit_g} with~\eqref{NTLimitsdzSL} and using once more~\eqref{ExpansionBesselK}, we find for the non-tangential limit at $x_{\Sigma}$  
  \begin{equation*}
    \begin{split}
      \lim_{\Omega_{\pm} \ni x \to x_\Sigma}& WL(w) \varphi( x) = \pm \frac{1}{4} \overline{{\bf n}(x_\Sigma)} \varphi(x_\Sigma) \\
      &\quad + \lim_{\varepsilon \to 0^{+}} \int_{\Sigma \setminus B(x_\Sigma,\varepsilon)} \frac{i \sqrt{w}}{4 \pi} \frac{ \overline{{\bf x}_\Sigma}-\overline{{\bf y}_\Sigma}}{|x_\Sigma-y_\Sigma|} K_1(- i \sqrt{w} |x_\Sigma-y_\Sigma|) \varphi(y_\Sigma) \dd \sigma(y_\Sigma).
    \end{split}
  \end{equation*}
  Since $WL(w) \varphi \in H^{1/2}_\Delta(\mathbb{R}^2 \setminus \Sigma)$, we conclude with \cite[Thm.~3.6~(ii)]{BGM23} that the latter non-tangential limit coincides with $\gamma_D^\pm (WL(w) \varphi)_\pm$. Therefore, we finally get that
\begin{equation*} 
\begin{split}
W(w) \varphi(x_\Sigma) &= \frac{1}{2} \big( \gamma_D^{+} (WL(w) \varphi)_{+} +  \gamma_D^{-} (WL(w) \varphi)_{-} \big)(x_{\Sigma})  \\
&= \lim_{\varepsilon \to 0^{+}} \int_{\Sigma \setminus B(x_\Sigma,\varepsilon)} \frac{i \sqrt{w}}{4 \pi} \frac{ \overline{{\bf x}_\Sigma}-\overline{{\bf y}_\Sigma}}{|x_\Sigma-y_\Sigma|} K_1(- i \sqrt{w} |x_\Sigma-y_\Sigma|) \varphi(y_\Sigma) \dd \sigma(y_\Sigma),
\end{split}
\end{equation*}
 which is the claimed identity.
\end{proof}

\subsection{Generalized boundary triples} \label{section_boundary_triples}
In this section we recall the concept of generalized boundary triples, which are a useful tool from the extension and spectral theory of symmetric and self-adjoint operators. Generalized boundary triples were introduced in  \cite{DM95} and used, e.g., in \cite{BEHL18, BFKLR17, PS18, BHSLS23, BMN17, HT23, MZ24} to study differential operators with boundary or transmission conditions. The results presented here are formulated in such a way that they can be applied directly in the situations we are considering; for further results and proofs we refer to \cite{BL07, DM95}. Many of the presented results are known for closely related ordinary or quasi boundary triples \cite{BHS20, BL07}, but not formulated for generalized boundary triples, and thus, we provide them, where necessary, with proofs. In the following $\mathcal{H}$ is always a Hilbert space. 

\begin{definition} \label{Definition_QBT_GBT}
Let $S$ be a closed and symmetric operator in the Hilbert space $\mathcal{H}$ and $A$ be an operator with $\overline{A} = S^{\ast}$. A triple $\{\mathcal{G} , \Gamma_0 , \Gamma_1\}$ is called generalized boundary triple for $S^\ast$, if $\mathcal{G}$ is a Hilbert space and $\Gamma_0 ,\Gamma_1 : \Dom(A) \to \mathcal{G}$ are linear mappings  which satisfy the following conditions:
\begin{enumerate}
\item[\textup{(i)}] The restriction $A_0 := A \upharpoonright \Ker(\Gamma_0)$ is self-adjoint in $\mathcal{H}$.

\item[\textup{(ii)}] $\Ran(\Gamma_0) = \mathcal{G}$.

\item[\textup{(iii)}] For all $f,g \in \Dom(A)$ there holds the abstract Green's identity 
\begin{equation} \label{GBT_Abstract_Green}
\langle A f , g \rangle_{\mathcal{H}} - \langle f , A g \rangle_{\mathcal{H}} = \langle \Gamma_1 f, \Gamma_0 g \rangle_{\mathcal{G}} - \langle \Gamma_0 f , \Gamma_1 g \rangle_{\mathcal{G}}.
\end{equation}
\end{enumerate}
\end{definition}

The next result, which is useful in applications, allows us to check whether a given operator $A$ and two boundary maps $\Gamma_0, \Gamma_1 : \Dom(A) \to \mathcal{G}$ form a generalized boundary triple for a suitable operator $S$. It can be seen as a variant of \cite[Thm.~2.3]{BL07} or \cite[Thm.~2.1.9]{BHS20} for generalized boundary triples.

\begin{proposition} \label{GBT_RateThm}
Let $A$ be an operator in $\mathcal{H}$. Furthermore, let  $\mathcal{G}$ be a Hilbert space and $\Gamma_0, \Gamma_1 : \Dom(A) \to \mathcal{G}$ be linear mappings such that the following holds:
\begin{enumerate}
\item[\textup{(i)}] $A \upharpoonright \Ker(\Gamma_0)$ is an extension of a self-adjoint operator $A_0$ in $\mathcal{H}$.

\item[\textup{(ii)}] $\Ran(\Gamma_0) =\mathcal{G}$ and $\Ker(\Gamma_0) \cap \Ker(\Gamma_1)$ is dense in $\mathcal{H}$.

\item[\textup{(iii)}] For all $f,g \in \Dom(A)$ there holds the abstract Green's identity 
\begin{equation*}
\langle A f , g \rangle_{\mathcal{H}} - \langle f , A g\rangle_{\mathcal{H}} = \langle \Gamma_1 f, \Gamma_0 g \rangle_{\mathcal{G}} - \langle \Gamma_0 f , \Gamma_1 g \rangle_{\mathcal{G}}.
\end{equation*}
\end{enumerate}
Then $S := A \upharpoonright \Ker(\Gamma_0) \cap \Ker(\Gamma_1)$ is a closed and symmetric operator in $\mathcal{H}$ such that $S^{\ast} = \overline{A}$. In particular, $\{\mathcal{G} , \Gamma_0 ,\Gamma_1\}$ is a generalized boundary triple for $S^{\ast}$ and $A_0 = A \upharpoonright \Ker(\Gamma_0)$.
\end{proposition}

\begin{proof}
First, by the abstract Green's identity in~(iii) the operator $A \upharpoonright \Ker(\Gamma_0)$ is symmetric in $\mathcal{H}$. Since $A \upharpoonright \Ker(\Gamma_0)$ is an extension of a self-adjoint operator $A_0$ by~(i), this can only be true if $A_0 = A \upharpoonright \Ker(\Gamma_0)$. 

Next, we show that the assumption $\Ran(\Gamma_0) = \mathcal{G}$ implies that $\Ran(\Gamma_0, \Gamma_1)$ is dense in $\mathcal{G} \times \mathcal{G}$. Indeed, as $A_0 = A \upharpoonright \Ker(\Gamma_0)$ is self-adjoint, we see that $\{ \mathcal{G}, (\Gamma_0, \Gamma_1), (\Gamma_0, \Gamma_1) \}$ is a triple for the dual pair $\{ \overline{S}, \overline{S} \}$ in the sense of \cite[Def.~2.1]{B23}. Hence, taking the assumptions (i)--(iii) into account, we can apply \cite[Lem.~2.3]{B23}, which shows that $\Ran(\Gamma_0, \Gamma_1)$ is dense in $\mathcal{G} \times \mathcal{G}$.

By the shown density of $\Ran(\Gamma_0, \Gamma_1)$  in $\mathcal{G} \times \mathcal{G}$ and the assumptions in~(i) and~(iii) we can apply \cite[Thm.~2.3]{BL07}, which implies that
 $S := A \upharpoonright \Ker(\Gamma_0) \cap \Ker(\Gamma_1)$ is a closed and symmetric operator in $\mathcal{H}$ such that $S^{\ast} = \overline{A}$. Thus, all assumptions in Definition~\ref{Definition_QBT_GBT} are fulfilled, $\{\mathcal{G} , \Gamma_0 ,\Gamma_1\}$ is a generalized boundary triple for $S^{\ast}$, and all claims are shown. 
\end{proof}

Next, the $\gamma$-field and the Weyl function of a given generalized boundary triple will be defined. These operator-valued functions are used in applications to study spectral properties of the self-adjoint extensions of the symmetric operator $S$. Let $\{ \mathcal{G}, \Gamma_0 , \Gamma_1 \}$ be a generalized boundary triple for $\overline{A} = S^{\ast}$ and $A_0 = A \upharpoonright \Ker(\Gamma_0)$. Then for each $w \in \rho(A_0)$ the direct sum decomposition 
\begin{equation}\label{Decomposition_DomA}
\Dom(A) = \Dom(A_0) \dot + \Ker(A-w) = \Ker(\Gamma_0) \dot + \Ker(A-w)
\end{equation}
is valid. In particular, the operator $\Gamma_0 \upharpoonright \Ker(A-w) : \Ker(A-w) \to \mathcal{G}$ is bijective. 

\begin{definition} \label{Definition_Gamma_Weyl}
Let $\{ \mathcal{G} , \Gamma_0 , \Gamma_1\}$ be a generalized boundary triple for $S^{\ast}$. Then the corresponding $\gamma$-field $\gamma$ and Weyl function $M$ are defined by
\begin{equation*}
\rho(A_0) \ni w \mapsto \gamma(w) := \big( \Gamma_0 \upharpoonright \Ker(A-w)  \big)^{-1} : \mathcal{G} \to \mathcal{H}
\end{equation*}
and 
\begin{equation*}
\rho(A_0) \ni w \mapsto M(w) := \Gamma_1 \big( \Gamma_0 \upharpoonright \Ker(A-w)  \big)^{-1}  : \mathcal{G} \to \mathcal{G},
\end{equation*}
respectively.
\end{definition}

The next result contains a number of useful properties of the $\gamma$-field and the Weyl function. The proof of these statements can be found in \cite[Prop.~2.6 and Sect.~3.2]{BL07} noting that each generalized boundary triple is also a quasi boundary triple in the sense of \cite{BL07}, see \cite[Sect.~3.2]{BL07}.

\begin{proposition} \label{Properties_Gamma_Weyl}
Let $\{ \mathcal{G} , \Gamma_0 , \Gamma_1\}$ be a generalized boundary triple for $\overline{A} = S^{\ast}$, $A_0 := A \upharpoonright \Ker(\Gamma_0)$, and denote by $\gamma$ and $M$ the associated $\gamma$-field and Weyl function, respectively. Then the following holds:
\begin{enumerate}
\item[\textup{(i)}]  $\gamma(w) : \mathcal{G} \to \mathcal{H}$ is bounded and $\Ran(\gamma(w)) \subseteq \Ker(A - w)$ for every $w \in \rho(A_0)$.

\item[\textup{(ii)}] For every $w \in \rho(A_0)$ one has $\gamma(w)^{\ast} = \Gamma_1 (A_0 - \overline{w})^{-1}$. 

\item[\textup{(iii)}]  $M(w) : \mathcal{G} \to \mathcal{G}$ is a linear and bounded operator for every $w \in \rho(A_0)$. Furthermore, for all $w , z \in \rho(A_0)$
\begin{equation*}
M(z) - M(w)^{\ast} =(z - \overline{w}) \gamma(w)^{\ast} \gamma(z)
\end{equation*}
holds. In particular, $M(w) = M(\overline{w})^{\ast}$ and $w \mapsto M(w)$ is holomorphic.
\end{enumerate}
\end{proposition}

Assume that $\{ \mathcal{G}, \Gamma_0, \Gamma_1 \}$ is a fixed generalized boundary triple for $\overline{A}  = S^*$, and let $B$ be a compact and self-adjoint operator in $\mathcal{G}$. In the following, we consider the extension 
\begin{equation} \label{def_A_B_abstract}
  A_B := A \upharpoonright \Ker(\Gamma_0 + B \Gamma_1)
\end{equation}
of $S$, i.e. $f \in \Dom(A_B)$ if and only if $f \in \Dom(A)$ satisfies the abstract boundary conditions $\Gamma_0 f + B \Gamma_1 f = 0$. In the next theorem we discuss how the spectral properties of $A_B$ can be described with the help of the $\gamma$-field and Weyl function associated with the triple $\{ \mathcal{G}, \Gamma_0, \Gamma_1 \}$. It can be seen as a variant of \cite[Thms.~2.6.1\&2.6.2]{BHS20} or \cite[Thm.~2.8]{BL07} for generalized boundary triples, and it is useful in the situation that $B$ admits a splitting $B = B_1 B_2$ with $B_1 : \mathcal{K} \to \mathcal{G}$ and $B_2 : \mathcal{G} \to \mathcal{K}$, where $\mathcal{K}$ is another Hilbert space.
In particular, item~(i) contains an abstract version of the Birman-Schwinger principle to characterize the  eigenvalues of $A_B$ and in assertion~(iii) a Krein type resolvent formula is stated.

\begin{theorem} \label{GBT_TB_SA}
Let $B  = B^{\ast} : \mathcal{G} \to \mathcal{G}$ be a compact operator, $\mathcal{K}$ be a Hilbert space and $B_1 : \mathcal{K} \to \mathcal{G}$ and $B_2 : \mathcal{G} \to \mathcal{K}$ be bounded operators such that $B = B_1 B_2$. Then, the operator $A_B$ defined by~\eqref{def_A_B_abstract} is self-adjoint in $\mathcal{H}$ and the following holds.
\begin{enumerate}
\item[\textup{(i)}] For every $w \in \rho(A_0)$ one has
\begin{equation} \label{Birman_Schwinger}
\Ker(A_B - w) = \gamma(w) B_1 \Ker (I + B_2 M(w) B_1).
\end{equation}
In particular, $w \in \sigma_{\textup{p}}(A_B)$ if and only if $0 \in \sigma_{\textup{p}}(I + B_2 M(w) B_1)$.

\item[\textup{(ii)}] For every $w \in \rho(A_0) \setminus \sigma_{\textup{p}}(A_B)$ the operator $I + B_2 M(w) B_1 : \mathcal{K} \to \mathcal{K}$ is boundedly invertible.

\item[\textup{(iii)}] There holds $\rho(A_0) \setminus \sigma_{\textup{p}}(A_B) = \rho(A_0) \cap \rho(A_B)$ and for every $w \in \rho(A_B) \cap \rho(A_0)$ the resolvent representation
\begin{equation*}
( A_B - w )^{-1} = ( A_0 - w )^{-1} - \gamma(w) B_1 \big( I + B_2 M(w) B_1 \big)^{-1} B_2 \gamma(\overline{w})^{\ast}
\end{equation*}
is valid.
\end{enumerate}
\end{theorem}

\begin{proof}
One could follow directly the arguments in the proof of, e.g., \cite[Thm.~2.8]{BL07}. However, since this special situation (where $B$ is decomposed as a product $B=B_1B_2$, where $B_1, B_2$ act in different Hilbert spaces) has not been, to the best of our knowledge, treated in the literature so far, we give a full proof here.

To show (i), assume first that $f \in \Ker(A_B - w) \setminus \{ 0 \}$ for some $w \in \rho(A_0)$. As the inclusion $\Ker(A_B - w) \subset \Ker(A - w)$ holds, there exists, by the definition of the $\gamma$-field, $\varphi \in \mathcal{G}$ such that $f = \gamma(w) \varphi$. Since $f$ also belongs to $\Dom(A_B)$, we have
\begin{equation} \label{Birman_Schwinger1}
  0 = (\Gamma_0 + B_1 B_2 \Gamma_1) f = \big( I + B_1 B_2 M(w) \big) \varphi \quad \Leftrightarrow \quad \varphi  = -B_1 B_2 M(w) \varphi.
\end{equation}
In particular, $\varphi \in \Ran(B_1)$, i.e. there exists $\psi_1 \in \mathcal{K}$ such that $\varphi = B_1 \psi_1$. Using this representation in~\eqref{Birman_Schwinger1}, we get that
\begin{equation*} 
  0 = B_1 \big( I + B_2 M(w) B_1  \big) \psi_1,
\end{equation*}
i.e. $\psi_2 := ( I + B_2 M(w) B_1  ) \psi_1 \in \Ker(B_1)$. Therefore, $\psi := \psi_1 - \psi_2$ satisfies $f = \gamma(w) \varphi = \gamma(w) B_1 \psi$, so in particular $\psi \neq 0$, and
\begin{equation*}
  \begin{split}
    \big( I + B_2 M(w) B_1  \big) \psi &= \big( I + B_2 M(w) B_1 \big) \psi_1 - \big( I + B_2 M(w) B_1  \big) \psi_2 \\
    &= \big( I + B_2 M(w) B_1 \big) \psi_1 - \psi_2 = 0,
  \end{split}
\end{equation*}
where $\psi_2 \in \Ker(B_1)$ was used in the second step. This shows the inclusion '$\subset$' in~\eqref{Birman_Schwinger}.

For the converse inclusion, assume that $\psi \in \Ker( I + B_2 M(w) B_1 ) \setminus \{ 0 \}$, which implies $B_1 \psi \neq 0$. Thus, $f := \gamma(w) B_1 \psi \neq 0$. Then, $f \in \Ker(A-w)$ and
\begin{equation*}
  (\Gamma_0 + B \Gamma_1)f = B_1 \psi + B_1 B_2 M(w) B_1 \psi = 0.
\end{equation*}
Therefore, $f \in \Ker(A-w) \cap \Dom (A_B) = \Ker (A_B - w)$, which shows that also the inclusion '$\supset$' in~\eqref{Birman_Schwinger} is true and finishes the proof of item~(i).

To show~(ii), assume that $w \in \rho(A_0) \setminus \sigma_{\textup{p}}(A_B)$ is fixed. Then by~\eqref{Birman_Schwinger} applied with $B_1 = I$ and $B_2=B$ one has $-1 \notin \sigma_{\textup{p}}(B M(w))$. Since $B$ is compact and $M(w)$ is bounded by Proposition~\ref{Properties_Gamma_Weyl}~(iii), it follows with \cite[Prop.~2.1.8]{P94} 
\begin{equation*}
\sigma(B_2 M(w) B_1 ) \setminus \{ 0\} = \sigma(B M(w)) \setminus \{0 \} = \sigma_{\textup{p}}(B M(w)) \setminus \{0 \}.
\end{equation*}
Consequently, $-1 \in \rho(B_2 M(w) B_1)$, i.e. $I + B_2 M(w) B_1 : \mathcal{K} \to \mathcal{K}$ is boundedly invertible.

Next, we will prove item~(iii). The inclusion $\rho(A_0) \setminus \sigma_{\textup{p}}(A_B) \supset \rho(A_0) \cap \rho(A_B)$ is evident. We will show the inclusion $\rho(A_0) \setminus \sigma_{\textup{p}}(A_B) \subset \rho(A_0) \cap \rho(A_B)$ and the resolvent formula at once. For this purpose, let $w \in \rho(A_0) \setminus \sigma_{\textup{p}}(A_B)$ be fixed. Then $A_B - w :\Dom(A_B) \to \mathcal{H}$ is injective. To show the surjectivity of this operator, we define for a given $f \in \mathcal{H}$ the element $g \in \mathcal{H}$ by
	\begin{equation*}
		g := (A_0 - w )^{-1} f - \gamma(w) B_1(I+B_2M(w)B_1)^{-1} B_2 \gamma(\overline{w})^\ast f.
	\end{equation*}
	Note that $g$ is indeed well-defined due to the bijectivity of $I + B_2 M(w) B_1 : \mathcal{K} \to \mathcal{K}$ shown above and that $g \in \Dom(A)$. The definition of the Weyl function and Proposition \ref{Properties_Gamma_Weyl} (ii) yield
	\begin{equation*}
		\begin{split}
		  \Gamma_0 g + B \Gamma_1 g  
		  &=  -B_1(I+B_2M(w)B_1)^{-1} B_2 \gamma(\overline{w})^\ast f + B \Gamma_1 (A_0 - w )^{-1} f \\
		  & \qquad  - B M(w) B_1(I+B_2M(w)B_1)^{-1} B_2 \gamma(\overline{w})^\ast f \\
		  &= 0,
        \end{split}
	\end{equation*}
which implies $g \in \Dom(A_B)$. Eventually, it follows from Proposition \ref{Properties_Gamma_Weyl} (i) that $(A_B - w) g = f$ holds. This yields the bijectivity of the operator $A_B - w : \Dom(A_B) \to \mathcal{H}$, i.e. $w \in \rho(A_0) \cap \rho(A_B)$, and the resolvent formula
	\begin{equation*}
		(A_B-w)^{-1}f = g = (A_0-w)^{-1}f- \gamma(w) B_1(I+B_2M(w)B_1)^{-1} B_2 \gamma(\overline{w})^\ast f 
	\end{equation*}
	for all $f \in \mathcal{H}$.

Finally, we will prove the self-adjointness of $A_B$. Since $B$ is self-adjoint, it follows from the abstract Green's  identity \eqref{GBT_Abstract_Green} that 
 	\begin{equation*}
 		\begin{split}
 			\langle A_B f , g \rangle_{\mathcal{H}} - \langle f , A_B g \rangle_{\mathcal{H}} &= \langle \Gamma_1 f, \Gamma_0 g \rangle_{\mathcal{G}} - \langle \Gamma_0 f , \Gamma_1 g \rangle_{\mathcal{G}} \\
 			&= - \langle\Gamma_1 f, B \Gamma_1 g \rangle_{\mathcal{G}} + \langle B\Gamma_1 f , \Gamma_1 g \rangle_{\mathcal{G}} = 0
 		\end{split}
 \end{equation*}
 is valid for all $f,g \in \Dom(A_B)$, i.e.  $A_B$ is symmetric.  Thus, $\sigma_{\textup{p}}(A_B) \subseteq \mathbb{R}$. Moreover, it was deduced in the previous paragraph that, for every $w \in \rho(A_0) \setminus \sigma_{\textup{p}}(A_B)$, $\Ran(A_B - w) =\mathcal{H}$. Choosing $w=\pm i$, we conclude that $A_B$ is self-adjoint.
\end{proof}

\section{Definition and spectral properties of $T_B$} \label{section_T_B}

In this section we rigorously define and study the operator $T_B$ that is defined by \eqref{eq:def_TB_alt}. For this, we construct first in Section~\ref{section_GBT_T_B} a generalized boundary triple that is suitable to investigate $T_B$. With the help of this generalized boundary triple we show the self-adjointness of $T_B$ in Section~\ref{section_def_selfadjointness}. Finally, Section~\ref{section_spectrum} is devoted to the spectral analysis of $T_B$. Recall that $\Omega_+ \subset \mathbb{R}^2$ is a bounded open set with a Lipschitz continuous boundary $\partial \Omega_+ = \Sigma$ as described in Section~\ref{section_function_spaces}. Moreover, $\Omega_- = \mathbb{R}^2 \setminus \overline{\Omega_+}$ and hence $ \partial \Omega_- = \partial \Omega_+ = \Sigma$.

\subsection{A generalized boundary triple suitable for the study of $T_B$} \label{section_GBT_T_B}

Define in $L^2(\mathbb{R}^2)$ the operator $T$ by
\begin{equation} \label{def_T}
\begin{split}
T f &= ( - \Delta f_{+} ) \oplus ( - \Delta f_{-} ), \\
\text{Dom}(T) &= \left\{ f \in H_{\Delta}^{1/2}(\mathbb{R}^2 \setminus \Sigma) \: \big| \: \partial_{\overline{z}} f_\pm \in H^{1/2}(\Omega_\pm) \right\},
\end{split}
\end{equation}
where $H_{\Delta}^{1/2}(\mathbb{R}^2 \setminus \Sigma)$ is defined by~\eqref{def_H_s_Delta_Sigma}.
We note that, using $4 \dz \dzbar = \Delta$,  one can represent $\Dom(T)$ equivalently as 
\begin{equation*}
\text{Dom}(T) = \left\{ f \in H_{\dzbar}^{1/2}(\Omega_{+}) \oplus H_{\dzbar}^{1/2}(\Omega_{-}) \: \big| \: \dzbar f_{\pm} \in   H_{\dz}^{1/2}(\Omega_{\pm}) \right\},
\end{equation*}
where the spaces $H_{\dz}^{1/2}(\Omega_{\pm})$ and $H_{\dzbar}^{1/2}(\Omega_{\pm})$ are defined by \eqref{eq_Hdz_bar} and \eqref{eq_Hdz}.
Hence, taking Lemma~\ref{Trace_Extension} into account, we can define the maps $\Gamma_0, \Gamma_1 : \text{Dom}(T) \to L^2(\Sigma; \mathbb{C}^2)$ acting on $f \in \Dom(T)$ by
\begin{equation} \label{def_Gamma}
\Gamma_0 f:=2 \begin{pmatrix}
				 \bar{\bn} ( \gamma_D^{+} \dzbar f_{+} - \gamma_D^{-} \dzbar f_{-} ) \\
				\bn (\gamma_D^{+} f_{+} - \gamma_D^{-} f_{-} ) 
				\end{pmatrix} \quad \text{and} \quad
\Gamma_1 f:= \frac{1}{2} \begin{pmatrix}
				 \gamma_D^{+} f_{+} +  \gamma_D^{-} f_{-}   \\
				- \gamma_D^{+} \dzbar f_{+} - \gamma_D^{-} \dzbar f_{-} 
				\end{pmatrix},
\end{equation}
where $\bn = n_1 + i n_2$ with $n=(n_1,n_2)$ being the unit outward normal vector to $\Omega_+$.
In the following proposition we show that $\{ L^2(\Sigma; \mathbb{C}^2), \Gamma_0, \Gamma_1 \}$ is a generalized boundary triple for $\overline{T}$. Recall that for $w \in \mathbb{C} \setminus [0, +\infty)$ the operators $SL(w)$ and $WL(w)$ are defined by~\eqref{def_single_layer_potential} and~\eqref{def_Psi}, respectively, and that $S(w)$ and $W(w)$ are the boundary integral operators introduced in~\eqref{def_single_layer_boundary_integral_operator} and~\eqref{def_L}, respectively.

\begin{proposition} \label{IsQBT}
The triple $\{L^2(\Sigma; \mathbb{C}^2), \Gamma_0, \Gamma_1\}$ is a generalized boundary triple for $\overline{T}$ such that $T \upharpoonright \Ker ( \Gamma_0) = - \Delta$, where $-\Delta$ is the free Laplacian defined on $\Dom(-\Delta) = H^2(\mathbb{R}^2)$. Moreover, for all $w\in \mathbb{C}\setminus[0,+\infty)$ the following is true:
\begin{itemize}
  \item[\textup{(i)}] The values of the $\gamma$-field $\gamma(w): L^2(\Sigma; \mathbb{C}^2) \rightarrow L^2(\mathbb{R}^2)$ are acting on $\varphi = (\varphi_1, \varphi_2) \in L^2(\Sigma; \mathbb{C}^2)$ as
  \begin{equation*}
    \gamma(w) \varphi = SL(w) \varphi_1 + WL(w) \varphi_2.
  \end{equation*}

  \item[\textup{(ii)}] For $f \in L^2(\mathbb{R}^2)$ there holds the representation 
  \begin{equation*}
    \gamma(w)^{\ast} f = \begin{pmatrix}
SL(w)^{\ast} f \\
WL(w)^{\ast} f 
\end{pmatrix} =  \begin{pmatrix}
	  \gamma_D (- \Delta - \overline{w} )^{-1} f \\
	   - \gamma_D \dzbar (- \Delta - \overline{w} )^{-1} f
    \end{pmatrix}.
  \end{equation*}
  In particular, $\gamma(w)^{\ast}$ gives rise to a bounded operator $\gamma(w)^{\ast} : L^2(\mathbb{R}^2) \to H^{1}(\Sigma) \times H^{1/2}(\Sigma)$.

  \item[\textup{(iii)}] The values of the Weyl function $M(w): L^2(\Sigma; \mathbb{C}^2) \rightarrow L^2(\Sigma; \mathbb{C}^2)$ are acting on $\varphi = (\varphi_1, \varphi_2) \in L^2(\Sigma; \mathbb{C}^2)$ as
  \begin{equation*}
    M(w) \varphi = \begin{pmatrix}
	  S(w) & W(w) \\
	   W(\overline{w})^{\ast} & \frac{w}{4} S(w)
    \end{pmatrix} \begin{pmatrix} \varphi_1 \\ \varphi_2 \end{pmatrix}.
  \end{equation*}
\end{itemize}
\end{proposition}
\begin{proof}
First, we show that $\{L^2(\Sigma; \mathbb{C}^2), \Gamma_0, \Gamma_1\}$ is a generalized boundary triple, for which we apply Proposition~\ref{GBT_RateThm}.
We begin by proving abstract Green's  identity. Let $f,g \in \Dom(T)$ be fixed. With the integration by parts  formula \eqref{GreenH12}, applied for $f_{\pm} \in H_{\dzbar}^{1/2}(\Omega_{\pm})$ and $\dzbar g_{\pm} \in   H_{\dz}^{1/2}(\Omega_{\pm})$, and $\Delta = 4 \dz \dzbar$ one finds that
	 \begin{equation*}
	 	\begin{aligned}
	 	\langle \partial_{\overline{z}} f_\pm, \partial_{\overline{z}} g_\pm \rangle_{\Omega_\pm} &= \pm\frac{1}{2} \langle \bn \gamma_D^{\pm} f_\pm, \gamma_D^{\pm} \partial_{\overline{z}} g_\pm \rangle_\Sigma  - \langle  f_\pm, \partial_z \partial_{\overline{z} }g_\pm \rangle_{\Omega_\pm}\\
	 	&= \pm\frac{1}{2} \langle \bn \gamma_D^{\pm} f_\pm, \gamma_D^{\pm} \partial_{\overline{z}} g_\pm \rangle_\Sigma  + \frac{1}{4} \langle  f_\pm, - \Delta g_\pm \rangle_{\Omega_\pm}.
	 	\end{aligned}
	 \end{equation*}
  Analogously, \eqref{GreenH12} applied for $\dzbar f_{\pm} \in H_{\dz}^{s}(\Omega_{\pm})$ and $ g_{\pm} \in   H_{\dzbar}^{s}(\Omega_{\pm})$  yields, after a complex conjugation,
		 \begin{equation*}
			\langle \partial_{\overline{z}} f_\pm, \partial_{\overline{z}} g_\pm \rangle_{\Omega_\pm} =  \pm \frac{1}{2} \langle \overline{\bn} \gamma_D^{\pm} \partial_{\overline{z}}f_\pm, \gamma_D^{\pm} g_\pm \rangle_\Sigma  + \frac{1}{4} \langle  - \Delta  f_\pm, g_\pm \rangle_{\Omega_\pm}.
	\end{equation*}
	Hence,
	\begin{equation*}
		 \langle  - \Delta  f_\pm, g_\pm \rangle_{\Omega_\pm} - \langle  f_\pm,- \Delta g_\pm \rangle_{\Omega_\pm} = \pm 2\Big( \langle \bn \gamma_D^{\pm} f_\pm, \gamma_D^{\pm} \partial_{\overline{z}}  g_\pm  \rangle_\Sigma -  \langle \overline{\bn}  \gamma_D^{\pm} \partial_{\overline{z}} f_\pm , \gamma_D^{\pm} g_\pm \rangle_\Sigma \Big).
	\end{equation*}
	By adding this equation for $\Omega_+$ and $\Omega_-$, one gets that
	\begin{equation*}
	  \begin{split}
		 \langle  T  f, g \rangle_{\mathbb{R}^2} - \langle  f, T g \rangle_{\mathbb{R}^2} &= 2\Big( \langle \bn \gamma_D^{+} f_+, \gamma_D^{+} \partial_{\overline{z}}  g_+  \rangle_\Sigma -  \langle \overline{\bn}  \gamma_D^{+} \partial_{\overline{z}} f_+ , \gamma_D^{+} g_+ \rangle_\Sigma \Big) \\
		 &\quad - 2\Big( \langle \bn \gamma_D^{-} f_-, \gamma_D^{-} \partial_{\overline{z}}  g_-  \rangle_\Sigma -  \langle \overline{\bn}  \gamma_D^{-} \partial_{\overline{z}} f_- , \gamma_D^{-} g_- \rangle_\Sigma \Big).
      \end{split}
	\end{equation*}
	On the other hand, a direct calculation involving the definition of $\Gamma_0$ and $\Gamma_1$ shows that 
	\begin{equation*}
	  \begin{split}
  	  \langle  \Gamma_1 f, \Gamma_0 g \rangle_{\Sigma} - \langle  \Gamma_0 f, \Gamma_1 g \rangle_{\Sigma} &= 2\Big( \langle \bn \gamma_D^{+} f_+, \gamma_D^{+} \partial_{\overline{z}}  g_+  \rangle_\Sigma -  \langle \overline{\bn}  \gamma_D^{+} \partial_{\overline{z}} f_+ , \gamma_D^{+} g_+ \rangle_\Sigma \Big) \\
		 &\quad - 2\Big( \langle \bn \gamma_D^{-} f_-, \gamma_D^{-} \partial_{\overline{z}}  g_-  \rangle_\Sigma -  \langle \overline{\bn}  \gamma_D^{-} \partial_{\overline{z}} f_- , \gamma_D^{-} g_- \rangle_\Sigma \Big).
      \end{split}
	\end{equation*}
	Therefore, the last two displayed formulas imply that the abstract Green's identity is fulfilled.

  Next, as $C^{\infty}_0(\Omega_{+}) \oplus C^{\infty}_0(\Omega_{-}) \subseteq \Ker(\Gamma_0) \cap \Ker(\Gamma_1)$, the set  $\Ker(\Gamma_0) \cap \Ker(\Gamma_1)$ is dense in $L^2(\mathbb{R}^2)$. Moreover,  $-\Delta \subseteq T \upharpoonright \Ker (\Gamma_0)$.

  In the next step we show $\text{Ran}(\Gamma_0) = L^2(\Sigma; \mathbb{C}^2)$. Let $w \in \mathbb{C} \setminus [0, +\infty)$ and $\varphi = (\varphi_1, \varphi_2) \in L^2(\Sigma; \mathbb{C}^2)$ be fixed. Define 
  \begin{equation} \label{def_f_w}
    f_w := SL(w) \varphi_1 + WL(w) \varphi_2.
  \end{equation}
  Then, taking the mapping properties of $SL(w)$ and $WL(w)$ from Lemma~\ref{Properties_SL} (i) and Proposition~\ref{proposition_Psi}~(i) into account, one finds that $f_w \in \Dom(T)$. Moreover, using $\partial_{\overline{z}} SL(w) = \widetilde{WL}(w)$ with $\widetilde{WL}(w)$ defined by~\eqref{def_Xi}, cf. Proposition~\ref{proposition_Psi}~(ii), Lemma~\ref{Properties_SL} (ii),~Lemma~\ref{lemma_jump_relations}, and $\dzbar WL(w) \varphi_2 \in H^1(\mathbb{R}^2)$, see Proposition~\ref{proposition_Psi}~(i), we find that 
  \begin{equation*}
    2\overline{\bn} ( \gamma_D^{+} \dzbar f_{w,+} - \gamma_D^{-} \dzbar f_{w,-} ) = 2\overline{\bn} ( \gamma_D^{+} (\widetilde{WL}(w) \varphi_1)_+ - \gamma_D^{-} (\widetilde{WL}(w) \varphi_1)_- ) = \varphi_1.
  \end{equation*}
  In a similar way, as $SL(w) \varphi_1 \in H^1(\mathbb{R}^2)$, one concludes with Lemma~\ref{lemma_jump_relations} that
  \begin{equation*}
    2 \bn (\gamma_D^{+} f_{w,+} - \gamma_D^{-} f_{w,-} ) = 2\bn (\gamma_D^{+} (WL(w) \varphi_2)_+ - \gamma_D^{-} (WL(w) \varphi_2)_- ) = \varphi_2.
  \end{equation*}
  By combining the latter two displayed formulas, we find that
  \begin{equation} \label{range_Gamma_0}
    \Gamma_0 f_w = \varphi,
  \end{equation}
  which shows the claimed surjectivity of $\Gamma_0$. Thus, all assumptions in Proposition~\ref{GBT_RateThm} are fulfilled. Therefore, $\{L^2(\Sigma; \mathbb{C}^2), \Gamma_0, \Gamma_1\}$ is a generalized boundary triple for $\overline{T}$, and $-\Delta = T \upharpoonright \Ker (\Gamma_0)$.
  
  To show the claim in~(i), let $\varphi = (\varphi_1, \varphi_2) \in L^2(\Sigma; \mathbb{C}^2)$, $w \in \mathbb{C} \setminus [0, +\infty)$, and $f_w \in \Dom(T)$ be defined by~\eqref{def_f_w}. Then, by Lemma~\ref{Properties_SL}~(ii) and Proposition~\ref{proposition_Psi}~(iii) we have that $(-\Delta - w) f_{w, \pm} = 0$ in $\mathbb{R}^2 \setminus \Sigma$ and thus, $f_w \in \Ker(T - w)$. Furthermore, by~\eqref{range_Gamma_0} also $\Gamma_0 f_w = \varphi$ holds. Therefore, the definition of $\gamma(w)$ in Definition~\ref{Definition_Gamma_Weyl} implies that
  \begin{equation*}
    \gamma(w) \varphi = f_w = SL(w) \varphi_1 + WL(w) \varphi_2.
  \end{equation*}
  The claim in~(ii) is a direct consequence of Proposition \ref{Properties_Gamma_Weyl} (ii), $T \upharpoonright \Ker(\Gamma_0) = -\Delta$, and the definition of $\Gamma_1$.
  
  It remains to show the claimed formula for the Weyl function $M$ in~(iii). Let $\varphi = (\varphi_1, \varphi_2) \in L^2(\Sigma; \mathbb{C}^2)$ be fixed. By using the definition of $M(w)$, the formula for $\gamma(w)$ from~(i), and $SL(w) \varphi_1, \partial_{\overline{z}} WL(w) \varphi_2 \in H^1(\mathbb{R}^2)$ we obtain  that
  \begin{equation} \label{equation_Weyl_function}
    \begin{split}
      M(w) \varphi &= \Gamma_1 \gamma(w) \varphi \\
      &= \begin{pmatrix}
	\gamma_D SL(w) \varphi_1 + \frac{1}{2} \big( \gamma_D^{+} (WL(w) \varphi_2)_{+} +  \gamma_D^{-} (WL(w) \varphi_2 )_{-} \big) \\
	- \frac{1}{2} \big( \gamma_D^{+} \dzbar (SL(w) \varphi_1)_{+} +  \gamma_D^{-} \dzbar (SL(w) \varphi_1 )_{-} \big) -\gamma_D \partial_{\overline{z}} WL(w) \varphi_2
\end{pmatrix} \\
     &= \begin{pmatrix}
	S(w) \varphi_1 + \frac{1}{2} \big( \gamma_D^{+} (WL(w) \varphi_2)_{+} +  \gamma_D^{-} (WL(w) \varphi_2 )_{-} \big) \\
	- \frac{1}{2} \big( \gamma_D^{+} \dzbar (SL(w) \varphi_1)_{+} +  \gamma_D^{-} \dzbar (SL(w) \varphi_1 )_{-} \big) + \frac{w}{4} S(w) \varphi_2
\end{pmatrix},
    \end{split}
  \end{equation}
  where~\eqref{def_single_layer_boundary_integral_operator} and Proposition~\ref{proposition_Psi}~(i) were used in the last step. Moreover, by~\eqref{def_L} we have
  \begin{equation*}
    \frac{1}{2} \big( \gamma_D^{+} (WL(w) \varphi_2)_{+} +  \gamma_D^{-} (WL(w) \varphi_2 )_{-} = W(w) \varphi_2.
  \end{equation*}
  Eventually, employing Proposition~\ref{proposition_Psi}~(ii),~\eqref{def_tilde_L}, and Lemma~\ref{lemma_W}, we obtain that
  \begin{equation*}
    \begin{split}
      - \frac{1}{2} \big( \gamma_D^{+} \dzbar &(SL(w) \varphi_1)_{+} +  \gamma_D^{-} \dzbar (SL(w) \varphi_1 )_{-} \big) \\
      &= - \frac{1}{2} \big( \gamma_D^{+} (\widetilde{WL}(w) \varphi_1)_{+} +  \gamma_D^{-} (\widetilde{WL}(w) \varphi_1 )_{-} \big) 
      = -\widetilde{W}(w) \varphi_1 = W(\overline{w})^* \varphi_1.
    \end{split}
  \end{equation*}
  By using the last two displayed formulas in~\eqref{equation_Weyl_function}, we find that the claimed expression for $M(w)$ is true. This finishes the proof of the proposition.
\end{proof}

\begin{remark} \label{Better_Mapping_Properties}
If $\Sigma$ is a $C^{k}$-boundary with $k \geq 2$, the mapping properties of  $\gamma(w)$ and $M(w)$ improve. First, with Lemma \ref{Properties_SL} (i)\&(ii) and Proposition~\ref{proposition_Psi}~(i) one gets for any $s,t \in [0,k-1]$ that
\begin{equation*}
\gamma(w) : H^{s-1/2}(\Sigma) \times H^{t-1/2}(\Sigma) \to H^{\min\{s+1, t\}}_\Delta(\R^2 \setminus \Sigma)
\end{equation*}
is a well-defined and bounded operator.
Furthermore, as by Proposition~\ref{proposition_Psi} we have $\dzbar \gamma(w) \varphi = \widetilde{WL}(w) \varphi_1 - \frac{w}{4} SL(w) \varphi_2$, we get with Lemma \ref{Properties_SL} (i)\&(ii) and~Proposition~\ref{proposition_Psi}~(ii) for all   $s,t \in [0, k-1]$ that
\begin{equation*}
 \dzbar \gamma(w) : H^{s-1/2}(\Sigma) \times H^{t-1/2}(\Sigma) \to H^{\min\{s, t +1\}}_\Delta(\R^2 \setminus \Sigma)
\end{equation*}
is well-defined and continuous.
Similarly, it follows with~\eqref{S_C_k},~\eqref{W_C_k},~\eqref{W_tilde_C_k},  and $\widetilde{W}(w) = -W(\overline{w})^*$ due to Lemma \ref{lemma_W} for all $s,t \in [\frac{1}{2},k-1]$ that
\begin{equation*}
	M(w) : H^{s-1/2}(\Sigma) \times H^{t-1/2}(\Sigma) \to H^{\min\{s+1/2,t-1/2\}}(\Sigma) \times H^{\min\{s-1/2,t+1/2\}}(\Sigma)
\end{equation*}
is a well-defined and bounded operator.
\end{remark}

In the following lemma, we state a useful result about the range of $\Gamma_1$.

\begin{lemma} \label{DensityOfGamma1H2inL2}
  The set $\Gamma_1(H^2(\R^2))$ is dense in $L^2(\Sigma;\C^2)$.
\end{lemma}
\begin{proof} 
Let $\varphi \in  L^2(\Sigma; \mathbb{C}^2)$ with  $\langle \varphi , \Gamma_1 g \rangle_{\Sigma} = 0$ for all $g \in H^2(\mathbb{R}^2)$. Set $f := \gamma(-1) \varphi \in \text{Ker}(T+1)$, which implies $\Gamma_0 f = \varphi$. Due to the abstract Green's identity we find that
\begin{equation*}
\begin{split}
  0 = \langle \varphi , \Gamma_1 g \rangle_{\Sigma} &= \langle \Gamma_0 f , \Gamma_1 g \rangle_{\Sigma} - \langle \Gamma_1 f , \Gamma_0 g \rangle_{\Sigma} \\
  &= \langle f , T g \rangle_{\mathbb{R}^2} - \langle T f , g \rangle_{\mathbb{R}^2} = \langle f , - \Delta g \rangle_{\mathbb{R}^2} + \langle f , g \rangle_{\mathbb{R}^2} \quad \forall g \in H^2(\mathbb{R}^2),
\end{split}
\end{equation*}
where $\Gamma_0 g = 0$ was used. Because of the self-adjointness of the free Laplacian $-\Delta$, this yields $f \in H^2(\mathbb{R}^2)$ and $(-\Delta + 1) f = 0$, i.e. $f \in \Ker(- \Delta +1) = \{0\}$. This shows $\varphi = \Gamma_0 f = 0$ and hence,  the set $\{ \Gamma_1 g \: | \: g \in H^2(\mathbb{R}^2) \}$ is dense in $L^2(\Sigma; \mathbb{C}^2)$.
\end{proof}

\subsection{Definition and Self-adjointness of $T_B$}\label{section_def_selfadjointness}

For a compact and self-adjoint operator $B: L^2(\Sigma; \mathbb{C}^2) \rightarrow L^2(\Sigma; \mathbb{C}^2)$
one may rewrite the definition \eqref{eq:def_TB_alt}, using the boundary mappings~\eqref{def_Gamma}, as
\begin{equation} \label{eq:def_TB}
  \begin{split}
    T_B f &= (-\Delta f_+) \oplus (-\Delta f_-), \\
    \Dom(T_B) &= \big\{ f \in H^{1/2}_\Delta(\mathbb{R}^2 \setminus \Sigma) \,| \, \partial_{\overline{z}} f_\pm \in H^{1/2}(\Omega_\pm), \Gamma_0 f + B \Gamma_1 f = 0 \big\},
  \end{split}
\end{equation}
where $H^{1/2}_\Delta(\mathbb{R}^2 \setminus \Sigma)$ is defined by~\eqref{def_H_s_Delta_Sigma}. In this section, we will show that the distributional action of \eqref{eq:compact_formal} on $f\in\Dom(T)$ belongs to $L^2(\R^2)$ if and only if one has $f\in\Dom(T_B)$, i.e. the transmission conditions \eqref{eq:TC_sub}, which may be also written as $\Gamma_0 f+B\Gamma_1f=0$, are satisfied. This will justify our definition of $T_B$. Then, with the help of the generalized boundary triple $\{ L^2(\Sigma; \mathbb{C}^2), \Gamma_0, \Gamma_1 \}$, see Proposition~\ref{IsQBT}, we will show that $T_B$ is self-adjoint.

\begin{proposition} \label{PropFormalDiffExpression}
	Let $T$ be defined by~\eqref{def_T} and $\{L^2(\Sigma; \mathbb{C}^2), \Gamma_0, \Gamma_1 \}$ be the generalized boundary triple given by~\eqref{def_Gamma}. Moreover, let $B : L^2(\Sigma;\C^2) \to L^2(\Sigma;\C^2)$ be compact and self-adjoint, $b_1,b_2,\dots \in \mathbb{R}$  be its eigenvalues, and $\varphi_1, \varphi_2, \dots \in L^2(\Sigma; \mathbb{C}^2)$ be the corresponding orthonormal eigenfunctions. Finally, denote  for $j \in \{1,2\}$ and $n \in \N$ by $\varphi_{n}^j \in L^2(\Sigma)$ the $j$-th component of $\varphi_n$. Then, for $f \in \Dom(T)$ the expression $\widetilde{f}$ defined by
\begin{equation}\label{FormalDiffExpressionB}
\begin{split}
\widetilde{f} = - \Delta f 
 + \sum_{n=1}^{\infty} b_n \Big( &| \varphi_{n}^1 \delta_{\Sigma} \rangle \langle \varphi_{n}^1 \delta_{\Sigma} | +   | \varphi_{n}^1 \delta_{\Sigma} \rangle \langle \dz (\varphi_{n}^2 \delta_{\Sigma}) | \\
 &+  | \dz (\varphi_{n}^2 \delta_{\Sigma}) \rangle \langle \varphi_{n}^1 \delta_{\Sigma} | + | \dz (\varphi_{n}^2 \delta_{\Sigma}) \rangle  \langle \dz (\varphi_{n}^2 \delta_{\Sigma}) | \Big) f
\end{split}
\end{equation}
is a well-defined distribution and $\widetilde{f} \in L^2(\R^2)$ if and only if 
		\begin{equation*}
			\Gamma_0 f + B\Gamma_1f = 0 \qquad \text{in }L^2(\Sigma;\C^2).
	\end{equation*}
\end{proposition} 
\begin{proof}
	Let $f \in \Dom(T)$ be fixed and $\widetilde{f}$ be defined by~\eqref{FormalDiffExpressionB}. First, we verify that $\widetilde{f} \in \mathcal{D}'(\mathbb{R}^2)$. Clearly, $-\Delta f \in \mathcal{D}'(\mathbb{R}^2)$ and also every summand of the series belongs to $\mathcal{D}'(\mathbb{R}^2)$. Hence, it suffices to show that the series converges in $\mathcal{D}'(\mathbb{R}^2)$. For this, let  $g \in \mathcal{D}(\R^2)$ and $m \in \N$ be fixed. To shorten notation we write $f|_\Sigma$ and $\partial_{\overline{z}}f|_\Sigma$ instead of $\tfrac{1}{2}( \gamma_D^{+} f_{+} + \gamma_D^{-}f_{-})$ and $\tfrac{1}{2}(\gamma_D^{+} \partial_{\overline{z}} f_{+} + \gamma_D^{-} \partial_{\overline{z}} f_{-})$. Then, using the definitions from Section \ref{section_distributions} and the compactness of $B$ we get
		\begin{equation*}
			\begin{aligned}
				& \bigg(\sum_{n=1}^{m} b_n \Big( | \varphi_{n}^1 \delta_{\Sigma} \rangle \langle \varphi_{n}^1 \delta_{\Sigma} | +   | \varphi_{n}^1 \delta_{\Sigma} \rangle \langle \dz (\varphi_{n}^2 \delta_{\Sigma})\ | \\
				&\hspace{100 pt}+  | \dz (\varphi_{n}^2 \delta_{\Sigma}) \rangle \langle \varphi_{n}^1 \delta_{\Sigma} | + | \dz (\varphi_{n}^2 \delta_{\Sigma}) \rangle  \langle \dz (\varphi_{n}^2 \delta_{\Sigma}) | \Big) f,g\bigg)\\
				&=\sum_{n=1}^m b_n \big( (\varphi_{n}^1 \delta_\Sigma, g)  +  ( \partial_{z}(\varphi_{n}^2 \delta_\Sigma),g)  \big)\big( (\overline{\varphi_{n}^1 } \delta_\Sigma ,f ) + ( \partial_{\overline{z}}(\overline{ \varphi_{n}^2} \delta_\Sigma) , f)\big) \\
				&=\sum_{n=1}^m b_n \big( \langle  \overline{g}|_\Sigma, \varphi_{n}^1 \rangle_{\Sigma}  - \langle \partial_{\overline{z}}\overline{g}|_\Sigma, \varphi_{n}^2 \rangle_{\Sigma}  \big)\big( \langle\varphi_{n}^1, f|_\Sigma \rangle_{\Sigma} - \langle  \varphi_{n}^2 , \partial_{\overline{z}} f|_\Sigma \rangle_{\Sigma}\big) \\
				& = \sum_{n=1}^m b_n \langle  (\overline{g}|_\Sigma,- \dzbar \overline{g}|_\Sigma )^T, (\varphi_{n}^1,\varphi_{n}^2)^T\rangle_\Sigma \langle(\varphi_{n}^1,\varphi_{n}^2)^T,(f|_\Sigma,- \dzbar f|_\Sigma)^T\rangle_{\Sigma}\\
				& = \sum_{n=1}^m b_n \langle  (\overline{g}|_\Sigma,- \dzbar \overline{g}|_\Sigma )^T ,\varphi_n\rangle_\Sigma \langle\varphi_n,( f|_\Sigma,- \dzbar f|_\Sigma)^T\rangle_{\Sigma}\\
				& \hspace{30 pt}  \overset{m \to \infty}{\longrightarrow} \langle  (\overline{g}|_\Sigma,- \dzbar \overline{g}|_\Sigma)^T, B( f|_\Sigma,-\partial_{\overline{z}} f|_\Sigma)^T \rangle_\Sigma = \langle \Gamma_1 \overline{g}, B \Gamma_1 f \rangle_\Sigma.
			\end{aligned}
	\end{equation*}
	Hence, $\widetilde{f} \in \mathcal{D}'(\R^2)$ and 
	\begin{equation}\label{eq_formal_1}
		( \widetilde{f},g) = \langle -\Delta \overline{g}, f \rangle_{\R^2}+ \langle \Gamma_1 \overline{g}, B \Gamma_1 f \rangle_\Sigma .
	\end{equation}
	If $ \Gamma_0 f + B\Gamma_1 f=0$, then we have by the abstract Green's identity and $\Gamma_0 g = 0$ 
	\begin{equation*}
		( \widetilde{f},g) = \langle- \Delta \overline{g},f \rangle_{\R^2}- \langle \Gamma_1 \overline{g} ,\Gamma_0 f \rangle_\Sigma = \langle \overline{g}, Tf\rangle_{\R^2}= (T f,g).
	\end{equation*}
	This holds for all $g \in \mathcal{D}(\R^2)$ and therefore  $\widetilde{f} = Tf \in L^2(\R^2)$. Conversely, if $\widetilde{f} \in L^2(\R^2)$, then it immediately follows that $\widetilde{f}_\pm = -\Delta f_\pm $ and hence $Tf = \widetilde{f}$. Thus, the abstract Green's identity yields
	\begin{equation*}
		(\widetilde{f},g) = \langle \overline{g}, Tf\rangle_{\R^2} = \langle -\Delta \overline{g}, f\rangle_{\R^2} - \langle \Gamma_1 \overline{g}, \Gamma_0 f \rangle_\Sigma \quad \forall g \in \mathcal{D}(\R^2).
	\end{equation*}
	Comparing this with \eqref{eq_formal_1} gives us
	\begin{equation*}
		\langle \Gamma_1 \overline{g}, \Gamma_0 f + B\Gamma_1 f \rangle_\Sigma = 0 \qquad \forall g \in \mathcal{D}(\R^2).
	\end{equation*}
	By continuity and density this equation remains true for $g \in H^2(\R^2)$. Hence, Lemma~\ref{DensityOfGamma1H2inL2} implies $\Gamma_0f + B \Gamma_1 f =0$.
\end{proof}

In the following theorem we discuss the self-adjointness of $T_B$ and some of its basic properties. Moreover, we state a Birman-Schwinger principle to characterize eigenvalues of $T_B$ in~(i) and a variant of Krein's resolvent formula for $T_B$ in item~(ii). Recall that $\gamma$ and $M$ denote the $\gamma$-field and the Weyl function associated with the generalized boundary triple $\{ L^2(\Sigma; \mathbb{C}^2), \Gamma_0, \Gamma_1 \}$; cf. Proposition~\ref{IsQBT}. Then the following result follows immediately from Theorem~\ref{GBT_TB_SA}.

\begin{theorem} \label{TBSelfadjoint}
Let $B$ be a compact and self-adjoint operator in $L^2(\Sigma; \mathbb{C}^2)$ and assume that $B = B_1 B_2$ for two bounded operators $B_1 : \mathcal{K} \to L^2(\Sigma; \C^2)$ and $B_2:L^2(\Sigma; \C^2) \to \mathcal{K}$ and a suitable Hilbert space $\mathcal{K}$. Then, the operator  $T_B$ defined by~\eqref{eq:def_TB} is self-adjoint in $L^2(\mathbb{R}^2)$ and the following statements hold:

\begin{enumerate}
\item[\textup{(i)}] For $w \in \mathbb{C} \setminus [0, +\infty)$ there holds
\begin{equation*}
\Ker(T_B - w) = \gamma(w) B_1 \Ker(I + B_2 M(w) B_1 ).
\end{equation*}
In particular, $w \in \sigma_{\textup{p}}(T_B)$ if and only if $0 \in \sigma_{\textup{p}}(I + B_2 M(w) B_1)$.
\item[\textup{(ii)}] For any $w \in \rho(T_B) \cap ( \mathbb{C} \setminus [0, +\infty) )$ the map $I + B_2 M(w) B_1$ is continuously invertible in $\mathcal{K}$ and there holds 
\begin{equation*}
	( T_B - w )^{-1} = ( - \Delta - w )^{-1} - \gamma(w) B_1 ( I + B_2 M(w) B_1 )^{-1} B_2 \gamma(\overline{w})^{\ast}.
\end{equation*}
\end{enumerate}
\end{theorem}

By choosing $B_1 = I$ and $B_2 = B$ in Theorem~\ref{TBSelfadjoint} we get the standard form of the Birman-Schwinger principle and Krein's resolvent formula for $T_B$. Moreover, if $B$ is compact in $H^t(\Sigma; \mathbb{C}^2)$ for some suitable $t$, then the Sobolev regularity in $\Dom(T_B)$ improves. We remark that the latter is the case, if $B$ is a finite rank operator and $\Ran(B) \subseteq H^t(\Sigma; \mathbb{C}^2)$.

\begin{corollary} \label{corollary_TBSelfadjoint}
Let $B$ be a compact and self-adjoint operator in $L^2(\Sigma; \mathbb{C}^2)$ and $T_B$ be defined by~\eqref{eq:def_TB}. Then the following statements hold:

\begin{enumerate}
\item[\textup{(i)}] For $w \in \mathbb{C} \setminus [0, +\infty)$ there holds
\begin{equation*}
\Ker(T_B - w) = \gamma(w) \Ker(I + B M(w)).
\end{equation*}
In particular, $w \in \sigma_{\textup{p}}(T_B)$ if and only if $0 \in \sigma_{\textup{p}}(I + B M(w))$.
\item[\textup{(ii)}] For any $w \in \rho(T_B) \cap ( \mathbb{C} \setminus [0, +\infty) )$ the map $I + B M(w)$ is continuously invertible in $L^2(\Sigma; \mathbb{C}^2)$ and  there holds 
\begin{equation*}
( T_B - w )^{-1} = ( - \Delta - w )^{-1} - \gamma(w) ( I + B M(w) )^{-1} B \gamma(\overline{w})^{\ast}.
\end{equation*}

\item[\textup{(iii)}] Assume that $B : H^s(\Sigma; \mathbb{C}^2) \to H^t(\Sigma; \mathbb{C}^2)$ is bounded for some $s \in [0,\frac{1}{2}],\, t\in[s,\frac{3}{2}]$, that $B$ is compact in $H^t(\Sigma; \mathbb{C}^2)$, and that $\Sigma$ is a $C^3$-boundary. Then
\begin{equation*}
\Dom(T_B) \subseteq \left\{ f \in H^{t+1/2}_{\Delta}(\mathbb{R}^2 \setminus \Sigma) \: \big| \: \dzbar f_{\pm} \in H^{t+1/2}(\Omega_{\pm}) \right\}.
\end{equation*}
\end{enumerate}
\end{corollary}

\begin{proof}
The claims in~(i) and~(ii) follow immediately from Theorem~\ref{TBSelfadjoint}, so it remains to prove item~(iii). Let $w \in \mathbb{C} \setminus \mathbb{R}$ be fixed. Taking the assumptions on $B$ and the mapping properties of $M(w)$ from Remark~\ref{Better_Mapping_Properties} into account, one sees that $I + B  M(w) : H^t(\Sigma; \C^2) \to H^t(\Sigma; \C^2)$ is a Fredholm operator with index $0$. Since $T_B$ is self-adjoint by Theorem~\ref{TBSelfadjoint}, it follows with~(i) that $I + B  M(w)$ is injective. Therefore, $I + B  M(w)$ is bijective in $H^t(\Sigma; \mathbb{C}^2)$. With Proposition~\ref{IsQBT}~(ii) and Remark~\ref{Better_Mapping_Properties} this yields that
\begin{equation*}
\Ran\big( \gamma(w) \big( I + B M(w) \big)^{-1} B \gamma(\overline{w})^* \big) \subseteq H^{t+1/2}_\Delta(\R^2 \setminus \Sigma)
\end{equation*}
and
\begin{equation*}
\Ran\big( \dzbar \gamma(w) \big( I + B M(w) \big)^{-1} B \gamma(\overline{w})^* \big) \subseteq H^{t+1/2}_\Delta(\R^2 \setminus \Sigma),
\end{equation*}
which together with $\Ran(-\Delta-w)^{-1}=H^2(\R^2)$ shows the claimed inclusion in~(iii) and finishes the proof.
\end{proof}

\subsection{Spectrum of $T_B$} \label{section_spectrum}

In this subsection we analyze the spectrum of $T_B$; cf. Theorem~\ref{SpectrumTB}. In particular, we will prove that $T_B$ is unbounded from below, if $B : L^2(\Sigma; \mathbb{C}^2) \to L^2(\Sigma; \mathbb{C}^2) $, given by 
 \begin{equation} \label{Decomposition_B}
	B = \begin{pmatrix}	B_{11}& B_{12} \\ B_{12}^\ast & B_{22}
	\end{pmatrix},
\end{equation}
is a compact and non-negative operator and  $B_{22}$ has infinitely many positive eigenvalues. For this, some preliminary considerations are necessary. In the next two propositions we analyze the limit of the Weyl function $M(w)$ associated with the generalized boundary triple $\{ L^2(\Sigma; \mathbb{C}^2), \Gamma_0, \Gamma_1 \}$, as the spectral parameter $w$ tends to $-\infty$. The first proposition deals with the convergence of $\sqrt{|w|}S(w)$, where $S(w)$ is the single layer boundary integral operator defined by~\eqref{def_single_layer_boundary_integral_operator} viewed as an operator in $L^2(\Sigma)$. To prove the result, we additionally assume that $\Sigma$ is a $C^2$-smooth loop, i.e. that $\Sigma$ has a $C^2$-smooth parametrization $\zeta$ as described at the beginning of Section~\ref{section_function_spaces}.

\begin{proposition}\label{lem_convergence_S}
	Let $\Sigma$ be a $C^2$-smooth loop. Then $\sqrt{|w|}S(w) : L^2(\Sigma) \to L^2(\Sigma)$ is uniformly bounded  with respect to $w \in (-\infty,0)$ and  converges in the strong sense to $\tfrac{1}{2}I$ for $w \to -\infty$.
\end{proposition}

\begin{proof}
   Let $L$ be the length of $\Sigma$ and $\zeta$ be an arc-length parametrization which satisfies \eqref{eq_bi_Lipschitz}.
	According to \cite[Thm.~1.3]{GS14} the norm of $S(w) : L^2(\Sigma) \to L^2(\Sigma)$, $w \in (-\infty,0)$, can be estimated by 
	\begin{equation*}
		\| S(w)\| \leq  \frac{C}{(2+|w|)^{1/2}} \ln(2+|w|^{-1})^{1/2},
	\end{equation*} 
	where $C>0$ does not depend on $w \in (-\infty,0)$.
	This implies the uniform boundedness of $\sqrt{|w|}S(w)$. Hence, it remains to prove the claim concerning the strong convergence. Since $\sqrt{|w|}S(w)$ is uniformly bounded and $C(\Sigma)$ is dense in $L^2(\Sigma)$, see  \cite[Thm.~3.14]{R87}, it suffices by \cite[Lem.~III.3.5]{kato} to prove 
	\begin{equation*}
		\lim_{w \to - \infty} \sqrt{|w|}S(w) \varphi = \frac{1}{2} \varphi  \quad \textup{in } L^2(\Sigma)
	\end{equation*} 
	for all $\varphi \in C(\Sigma)$. Let  $\varphi \in C(\Sigma)$ be fixed. First, we show  that $\sqrt{|w|}S(w)\varphi $ converges pointwise to $\varphi/2$. With~\eqref{IntegralRepresSLBIO} one sees for $t \in [0,L]$ that
	\begin{equation*}
		\begin{aligned}
			(\sqrt{|w|}S(w)\varphi)(\zeta(t)) &= \frac{1}{2 \pi}\int_{t-L/2}^{t+L/2} \sqrt{|w|}K_0(\sqrt{|w|}|\zeta(t) - \zeta(s)|) \varphi(\zeta(s)) \dd s \\
			&=g_{w,1}(t) + g_{w,2}(t) 
		\end{aligned}
	\end{equation*}
	with $g_{w,1}, g_{w,2} : [0, L] \to \mathbb{C}$ defined by
	\begin{equation*}
		\begin{aligned}
			g_{w,1}(t)&= \frac{1}{2\pi} \int_{B(t, \ln|w| / \sqrt{|w|})} \sqrt{|w|}K_0(\sqrt{|w|}|\zeta(t) - \zeta(s)|) \varphi(\zeta(s)) \dd s, \\
			g_{w,2}(t)&= (\sqrt{|w|}S(w)\varphi)(\zeta(t))  - g_{w,1}(t), 
		\end{aligned}
	\end{equation*}
where we used that our choice of the complex square root yields $-i \sqrt{w} = \sqrt{|w|}$ for $w<0$.
	Next, we consider the pointwise convergence of $g_{w,2}$. Let $t \in [0,L]$ be fixed. If $|w|>0$ is sufficiently large, then $ B(t, \ln|w| / \sqrt{|w|}) \subseteq B(t, L/2)$. Consequently,
	\begin{equation*}
		g_{w,2}(t)= \frac{1}{2\pi} \int_{B(t,L/2)\setminus B(t, \ln|w| / \sqrt{|w|})} \sqrt{|w|}K_0(\sqrt{|w|}|\zeta(t) - \zeta(s)|) \varphi(\zeta(s)) \dd s.
	\end{equation*}
	Due to Lemma~\ref{lemma_Bessel_functions}~(i)\&(iv) there exists a constant $C>0$ such that the estimate $0<K_0(\upsilon) \leq C \ee^{-\upsilon}$ is valid for all $\upsilon>1$. Furthermore, \eqref{eq_bi_Lipschitz} implies for all $s \in B(t,L/2)\setminus B(t, \ln|w| / \sqrt{|w|})$ and $|w|$ sufficiently large
	\begin{equation*}
		\sqrt{|w|}|\zeta(t) - \zeta(s)| \geq 	C_\zeta  \sqrt{|w|}| t -s | \geq C_\zeta  \ln|w| >1.
	\end{equation*}
	This and  estimations of the integrals lead to 
	\begin{equation*}
		\begin{aligned}
			|g_{w,2}(t)| &\leq C\|\varphi\|_{L^\infty(\Sigma)}\int_{B(t,L/2)\setminus B(t, \ln|w| / \sqrt{|w|} )} \sqrt{|w|} \ee^{-C_\zeta  \sqrt{|w|}| t -s | } \dd s\\
			&\leq  C\|\varphi\|_{L^\infty(\Sigma)}  \int_{ \ln |w| / \sqrt{|w|}}^\infty \sqrt{|w|}\ee^{-C_\zeta  \sqrt{|w|}s} \dd s\\
			&=   C\|\varphi\|_{L^\infty(\Sigma)}  \int_{ \ln |w|}^\infty \ee^{-C_\zeta s} \dd s.
		\end{aligned}
	\end{equation*}
	Hence, $g_{w,2}(t) \rightarrow 0$ for $w \to -\infty$. To compute the limit of $g_{w,1}(t)$, we write $g_{w,1}(t) = g_{w,3}(t) + g_{w,4}(t)$ with $g_{w,3}, g_{w,4} : [0,L] \to \mathbb{C}$ defined by 
	\begin{equation*}
		\begin{aligned}
			g_{w,3}(t) &= \frac{1}{2\pi} \int_{B(t, \ln|w| /\sqrt{|w|})} \sqrt{|w|}K_0(\sqrt{|w|}|t-s|) \varphi(\zeta(s)) \dd s \\
			g_{w,4}(t) &= \frac{1}{2\pi} \int_{B(t, \ln|w| / \sqrt{|w|})} \sqrt{|w|}\Big(K_0(\sqrt{|w|}|\zeta(t) - \zeta(s)|) \\
			&\qquad \qquad \qquad \qquad \qquad \qquad  -K_0(\sqrt{|w|}|t-s|) \Big)\varphi(\zeta(s)) \dd s.
		\end{aligned}
	\end{equation*}
	To analyze the behaviour of $g_{w,4}(t)$ for $w \to - \infty$, we introduce for $\vartheta \in [0,1]$ the function 
	\begin{equation*}
		\zeta_{\vartheta}(t,s) = \vartheta|\zeta(t) - \zeta(s)| + (1-\vartheta) |t-s|.
	\end{equation*}
	Then, by the fundamental theorem of calculus one has
	\begin{equation} \label{RepresGW4}
		\begin{split}
			g_{w,4}(t) = \frac{1}{2\pi}&\int_{B(t, \ln|w| / \sqrt{|w|})} |w| \\
			&\qquad \int_0^1 K_0'(\sqrt{|w|} \zeta_\vartheta(t,s)) (\zeta_1(t,s) - \zeta_0(t,s)) \, d \vartheta \varphi(\zeta(s))\dd s. 
		\end{split}
	\end{equation}
	Since $\zeta$ is an $L$-periodic $C^2$-function with $|\dot{\zeta}(t)| =1$, one finds with Taylor's theorem that there exists a constant $C>0$ such that  
	\begin{equation}\label{estimate_gamma_10}
		|\zeta_1(t,s) - \zeta_0(t,s)| \leq C |t-s|^2 \quad \forall s \in B(t, \ln|w| / \sqrt{|w|}).
	\end{equation}
	Furthermore, as $B(t,\ln|w| / \sqrt{|w|}) \subseteq B(t,L/2)$ for sufficiently large $|w|>0$, \eqref{eq_bi_Lipschitz} implies
	\begin{equation} \label{BoundGammaTau}
		C_\zeta |t-s| \leq \zeta_\vartheta(t,s) \leq  |t-s| \quad \forall s \in B\big(t,\ln|w| /\sqrt{|w|}\big), \vartheta \in [0,1].
	\end{equation}
	The recurrence relations and asymptotics of the modified Bessel functions in Lemma~\ref{lemma_Bessel_functions} yield  that there exists a  constant $C >0$ such that 
		\begin{equation} \label{BoundDervK0}
			|K_0'(\upsilon)| = K_1(\upsilon) \leq C \left ( \frac{\ee^{- \upsilon}}{\upsilon} + \frac{\ee^{- \upsilon}}{\sqrt{\upsilon}} \right) \leq C \frac{\ee^{-\tfrac{\upsilon}{2}}}{\upsilon}, \quad \forall \upsilon \in (0,+\infty).
	\end{equation}
	Using \eqref{estimate_gamma_10}, \eqref{BoundGammaTau}, and \eqref{BoundDervK0} in \eqref{RepresGW4} leads to 
		\begin{equation*}
			\begin{aligned}
				|g_{w,4}(t)| &\leq C\|\varphi\|_{L^\infty(\Sigma)}  \int_{B(t, \ln|w| / \sqrt{|w|})} |w|(t-s)^2  \frac{\ee^{-\tfrac{ C_\zeta}{2} \sqrt{|w|}|t-s|}}{\sqrt{|w|}|t-s|} \dd s\\
				& = C\|\varphi\|_{L^\infty(\Sigma)}  \frac{1}{\sqrt{|w|}}\int_{B(0, \ln|w| )} |s|  \ee^{-\tfrac{ C_\zeta}{2} |s|} \dd s \\
				& \leq C\|\varphi\|_{L^\infty(\Sigma)}  \frac{1}{\sqrt{|w|}}\int_0^\infty s  \ee^{-\tfrac{ C_\zeta}{2} s} \dd s.
			\end{aligned}
	\end{equation*} 
	Hence, $g_{w,4}(t) \rightarrow 0$ for $w \to -\infty$. To complete the proof of the pointwise convergence $\sqrt{|w|} S(w)\varphi \to \varphi/2$, it remains to investigate  $g_{w,3}(t)$ for $w \to - \infty$. Due to  $0< K_0(\upsilon)$  for every $\upsilon>0$ and the continuity of $\varphi \circ \zeta $, the mean value theorem for integrals yields the existence of an $s_{t,w} \in B\big(t, \ln|w|  /\sqrt{|w|}\big)$ such that 
	\begin{equation*}
		\begin{aligned}
			g_{w,3}(t) &= \varphi(\zeta(s_{t,w}))\frac{1}{2\pi}\int_{B\big(t, \ln|w| / \sqrt{|w|} \big)} \sqrt{|w|}K_0(\sqrt{|w|}|t-s|)  \dd s \\
			&= \varphi(\zeta(s_{t,w})) \frac{1}{\pi} \int_{0}^{\ln|w|} K_0(s) \dd s.
		\end{aligned}
	\end{equation*}
	Consequently, with 
	\begin{equation} \label{IntegralBesselK0}
		\int_0^{\infty} K_0(s )  \dd s = \frac{\pi}{2}
	\end{equation}
	we conclude that $g_{w,3}(t)  \to \frac{1}{2} \varphi(\zeta(t))$ for $w \to - \infty$.  Since this holds for all $t \in [0,L]$, 	$(\sqrt{|w|}S(w)\varphi)\circ \zeta$ converges pointwise to $\tfrac{1}{2}\varphi \circ \zeta$, which in particular yields the  pointwise convergence of $\sqrt{|w|}S(w)\varphi$ to $\varphi/2$ for $w \to - \infty$.

	To finish the proof, we show that $\sqrt{|w|} S(w) \varphi \rightarrow \frac{1}{2} \varphi$ also in $L^2(\Sigma)$. For this, note  that $K_0$ is strictly monotonically decreasing on $(0, +\infty)$ and hence, using~\eqref{eq_bi_Lipschitz} we find for all $t \in \mathbb{R}$
	\begin{equation*}
		\begin{aligned}
			\big|\sqrt{|w|}S(w)\varphi(\zeta(t))\big| &\leq  \frac{\sqrt{|w|}}{2\pi} \int_{t-L/2}^{t+L/2} K_0(\sqrt{|w|}C_\zeta |t-s|) |\varphi(\zeta(s))| \dd s\\
			&\leq \| \varphi \|_{L^\infty(\Sigma)}  \frac{1}{C_\zeta \pi} \int_0^\infty K_0(s) \dd s =  \frac{\| \varphi \|_{L^\infty(\Sigma)} }{2 C_\zeta},
		\end{aligned}
	\end{equation*}
	where \eqref{IntegralBesselK0} was used in the last equality. Therefore, the pointwise upper bound $|\sqrt{|w|}S(w)\varphi(x_\Sigma)| \leq    \| \varphi \|_{L^\infty(\Sigma)} / (2 C_\zeta)$ holds for all $x_\Sigma \in \Sigma$. Thus, an application of the dominated convergence theorem  shows that $\sqrt{|w|} S(w) \varphi$ converges to $\frac{1}{2} \varphi$ in $L^2(\Sigma)$, which finishes the proof.
\end{proof}

By \cite[Lem.~6.9~(ii)]{T09} the composition of $S(w)$ with a compact operator turns the strong convergence in Proposition~\ref{lem_convergence_S} into convergence in the operator norm. Thus, we obtain immediately the following corollary.

\begin{corollary}\label{cor_SK_convergence}
	Additionally to the assumptions of Proposition \ref{lem_convergence_S}, suppose that $K$ is a compact operator in $L^2(\Sigma)$. Then, $\sqrt{|w|}S(w)K$ converges in the operator norm to $K/2$ for $w \to -\infty$. 	
\end{corollary}

In the next proposition we analyze the behaviour of $W(w)$ defined  in~\eqref{def_L}, as $w \rightarrow -\infty$. We point out that no $C^2$-assumption on $\Sigma$ is needed for this.

\begin{proposition}\label{lem_boundedness_T} 
Let	$\tau >0$. Then $|w|^{-\tau} W(w) \rightarrow 0$  in the operator norm for $w \to - \infty$.
\end{proposition}

\begin{proof}
	Let $w < 0$ be fixed. Recall that in this case our choice of the complex square root yields $-i\sqrt{w} = \sqrt{|w|} > 0$. Furthermore, by Lemma~\ref{lemma_W} $W(w)$ is a strongly singular integral operator with kernel $k_W(x_\Sigma - y_\Sigma)$, where
	\begin{equation*}
		k_W(x) = -\frac{1}{4\pi} \sqrt{|w|}K_1(\sqrt{|w|}|x|) \frac{\overline{{\bf x}} }{|x|}.
	\end{equation*}
	Next, let the function $k_{W_0}$ be defined by
	\begin{equation*}
		k_{W_0}(x) = -\frac{1}{4\pi}\frac{\overline{{\bf x}}}{|x|^2}=-\frac{1}{4\pi}\frac{1}{{\bf x}}, \quad x \in \mathbb{R}^2 \setminus \{ 0 \}.
	\end{equation*}
	Then, due to  \cite[p.~1071]{EFV92}, the  strongly singular integral operator
	$W_0:L^2(\Sigma) \to L^2(\Sigma)$ acting as
	\begin{equation*}
		W_0 \varphi(x_\Sigma) := \lim_{\varepsilon \rightarrow 0+} \int_{\Sigma \setminus B(x_\Sigma, \varepsilon)} k_{W_0}(x_\Sigma - y_\Sigma) \varphi(y_\Sigma) \dd \sigma(y_\Sigma), \quad \varphi \in L^2(\Sigma),~x_\Sigma \in \Sigma,
	\end{equation*}
	is well-defined and bounded. Let $\delta \in (0,\min\{2\tau,1\})$ be arbitrary but fixed. Then,  $1-\ee^{-\upsilon} \leq  \upsilon^{\delta}$ for all $\upsilon>0$ and Lemma~\ref{lemma_Bessel_functions} (iii)\&(v) yield for every $\upsilon >0$ the estimate
	\begin{equation*}
		\left|K_1(\upsilon) - \frac{1}{\upsilon} \right| \leq \left|K_1(\upsilon) -\frac{\ee^{-\upsilon}}{\upsilon} \right| + \left|\frac{1-\ee^{-\upsilon}}{\upsilon} \right| \leq C (\ee^{-\upsilon} + \upsilon^{\delta-1})
	\end{equation*}  
	with some constant $C>0$. In particular,  for  $k = k_{W} - k_{W_0}$ this implies for all $x \in \mathbb{R}^2 \setminus \{0\}$ the pointwise bound
	\begin{equation*}
			|k(x)| = \frac{\sqrt{|w|}}{4\pi} \left|K_1(\sqrt{|w|}|x|) - \frac{1}{\sqrt{|w|} |x|} \right| 
			\leq C \Big(\sqrt{|w|} \ee^{-\sqrt{|w|}|x|} + |w|^{\delta/2}|x|^{\delta-1}\Big).
	\end{equation*}
	Therefore, by Proposition~\ref{prop.-sing.int.op.on Sigma}, which can be used due to the assumptions on $\Sigma$,
	\begin{equation*}
		\| W(w) - W_0 \| \leq C \int_0^{C_\zeta L/2 } \sqrt{|w|} \ee^{-\sqrt{|w|}s} + |w|^{\delta/2}s^{\delta-1} \dd s \leq C (1+|w|^{\delta/2}).
	\end{equation*}
	Since $\delta  < 2 \tau$, this finally yields
	\begin{equation*}
		\begin{split}
			|w|^{-\tau} \| W(w) \| &\leq |w|^{-\tau}  \| W_0 \| + |w|^{-\tau} \| W(w) - W_0 \| \\
			&\leq |w|^{-\tau} \| W_0 \|+  C \big(|w|^{-\tau}+|w|^{\delta/2 - \tau}\big) \rightarrow 0
		\end{split}
	\end{equation*}
	for $w \rightarrow -\infty$, which is the claim of this proposition.
\end{proof}

To be able to use the convergence results from Propositions~\ref{lem_convergence_S} and~\ref{lem_boundedness_T} for the spectral analysis of $T_B$, we will use the following lemma about the convergence of eigenvalues of compact operators that converge in the operator norm. This result follows from \cite[\S~IV.3.5]{kato}, see also  \cite[\S~II.5.2 and \S~III.6.5]{kato} for further explanations.

\begin{lemma}\label{lem_convergence_ev}
	Let $(K_n)_{n  \in \N}$ be a sequence of compact and self-adjoint operators in an infinite dimensional Hilbert space $\mathcal{H}$, which converge in the operator norm to a non-negative operator $K$. Moreover, let $(\lambda_k)_{k \in \N}$ and $(\lambda_k^n)_{k \in \N}$ be the decreasing sequences of non-negative eigenvalues of $K$  and $K_n$, respectively, possibly extended by zero and taking their multiplicities into account. Then for every $k \in \N$ there holds
	\begin{equation*}
		\lim_{n \to \infty} \lambda_k^n =  \lambda_k.
	\end{equation*}
\end{lemma}

Now, we are prepared to prove the main result about the spectrum of $T_B$ with compact and self-adjoint parameters $B$.

\begin{theorem} \label{SpectrumTB} 
Let $B$ be a compact and self-adjoint operator in $L^2(\Sigma; \mathbb{C}^2)$ and $T_B$ be defined by~\eqref{eq:def_TB}. Then the following statements hold:
\begin{enumerate}
\item[\textup{(i)}] $\sigma_{\textup{ess}}(T_B) = [0, +\infty)$.

\item[\textup{(ii)}] If $B$ has finite rank $k = \textup{dim}(\textup{Ran}(B))$, then $T_B$ has at most $k$ discrete eigenvalues.

\item[\textup{(iii)}]  Let $B : L^2(\Sigma; \C^2) \to L^2(\Sigma; \C^2)$ be given by  \eqref{Decomposition_B} with  $B_{12} = B_{22} = 0$. Then $T_B$ has at most finitely many discrete eigenvalues and there holds 
\begin{equation} \label{Domain_B1}
\Dom(T_B) = \left\{ f \in H^{3/2}_{\Delta}(\R^2 \setminus \Sigma) \cap H^1(\R^2) \: \big| \: \gamma_N^{+} f_{+} - \gamma_N^{-} f_{-} = - B_{11} \gamma_D f \right\}.
\end{equation}

\item[\textup{(iv)}] Assume that $\Sigma$ is a $C^2$-boundary and let $B$ be non-negative and be given by \eqref{Decomposition_B}.  Furthermore, assume that $B_{22}$ has infinitely many positive eigenvalues. Then, $\sigma_{\textup{disc}}(T_B)$ is infinite and unbounded from below.
\end{enumerate}
\end{theorem}

\begin{proof}
	To show item~(i), we note first that  Corollary~\ref{corollary_TBSelfadjoint}~(ii) and the compactness of $B$ imply that
	\begin{equation} \label{krein_in_proof}
		( T_B-w )^{-1} - (-\Delta - w )^{-1} = - \gamma(w) ( I + B M(w) )^{-1} B \gamma(\overline{w})^{\ast}
	\end{equation}
	is compact in $ L^2(\mathbb{R}^2)$ for any $w \in \mathbb{C} \setminus \mathbb{R}$. Hence, Weyls Theorem \cite[Thm.~XIII.14]{RS77} yields 
	\begin{equation*}
		\sigma_{\text{ess}}(T_B) = \sigma_{\text{ess}}(-\Delta) = \sigma(-\Delta) = [0, +\infty).
	\end{equation*}

	To prove the assertion in~(ii), note that the operator in~\eqref{krein_in_proof}
	is a finite rank operator with rank at most $k$.
	Together with \cite[Chap.~9.3, Thm.~3]{BS87} this yields that $\sigma_{\text{disc}}(T_B)$ contains at most $k$ eigenvalues taking multiplicities into account.

Next, we prove item (iii). We first verify that~\eqref{Domain_B1} is true. Recall that by Corollary~\ref{corollary_TBSelfadjoint}~(ii) for any $w \in \mathbb{C} \setminus \mathbb{R}$ the operator  
\begin{equation*}
   I + B M(w) = \begin{pmatrix} I + B_{11} S(w) & B_{11} W(w) \\ 0 & I \end{pmatrix}
\end{equation*}
is bijective in $L^2(\Sigma; \mathbb{C}^2)$. In particular,  $I + B_{11} S(w)$ is boundedly invertible in $L^2(\Sigma; \mathbb{C})$. A simple computation shows that
\begin{equation*}
   \big(I + B M(w)\big)^{-1} = \begin{pmatrix} (I + B_{11} S(w))^{-1} & - (I + B_{11} S(w))^{-1}  B_{11} W(w) \\ 0 & I \end{pmatrix}.
\end{equation*}
Taking again $B_{12}=B_{22}=0$ and the last displayed formula into account, we find that Krein's resolvent formula in~\eqref{krein_in_proof} can be simplified to
\begin{equation*}
  ( T_B - w  )^{-1} = ( - \Delta - w )^{-1} - SL(w) \big( I + B_{11} S(w) \big)^{-1} B_{11} \gamma_D ( - \Delta - w )^{-1}.
\end{equation*}
With the mapping properties of $SL(w)$ from  Lemma~\ref{Properties_SL} (i)  we find that 
\begin{equation} \label{delta_domain_inclusion}
  \Dom (T_B) \subset H^{3/2}_{\Delta}(\R^2 \setminus \Sigma) \cap H^1(\R^2).
\end{equation}
Moreover, by definition, a function $f \in H^{3/2}_{\Delta}(\R^2 \setminus \Sigma) \cap H^1(\R^2)$ belongs to $\Dom(T_B)$ if and only if $\Gamma_0 f + B \Gamma_1 f=0$, which is, due to  $B_{12}=B_{22}=0$, equivalent to
\begin{equation} \label{transmission_delta1}
	2 \bar{\bn} ( \gamma_D^{+} \dzbar f_{+} - \gamma_D^{-} \dzbar f_{-} ) = -B_{11} \gamma_D f.
\end{equation}
To proceed, note that  there exists a linear and bounded tangential derivative operator $\partial_t : H^1(\Sigma) \to L^2(\Sigma)$ which acts on functions $g \in C_0^{\infty}(\overline{\Omega_{\pm}})$ as 
	\begin{equation*}
		\partial_t \gamma_D^{\pm} g = n_1 \gamma_D^{\pm} \partial_2 g - n_2 \gamma_D^{\pm} \partial_1 g,
	\end{equation*}
see, e.g., \cite[Lem. A.1]{GM08}.
	Next, let $\zeta$ be an arc-length parametrization of $\Sigma$ such that $n \circ \zeta = (\dot{\zeta_2}, -\dot{\zeta_1})$ and $t \circ \zeta = (\dot{\zeta_1}, \dot{\zeta_2})$ is the unit tangent vector.
	Then, by a direct calculation one finds for any $g \in C_0^{\infty}(\overline{\Omega_{\pm}})$ 
	\begin{equation*}
		\begin{split}
			2 (\bar{\bn} \gamma_D^{\pm} \dzbar g) \circ \zeta &= \dot{\zeta_2} \gamma_D^{\pm} (\partial_1 g) \circ \zeta - \dot{\zeta_1} \gamma_D^{\pm} (\partial_2 g)\circ \zeta \\
			&\qquad \qquad +i \dot{\zeta_1} \gamma_D^{\pm} (\partial_1 g)\circ \zeta +i \dot{\zeta_2} \gamma_D^{\pm} (\partial_2 g)\circ \zeta \\
			&= (\gamma_N^\pm g + i \partial_t \gamma_D^\pm g) \circ \zeta.
		\end{split}
	\end{equation*} 
	By density, this identity extends to all functions $g \in H^{3/2}_{\Delta}(\Omega_{\pm})$. In particular, for any $f \in H^{3/2}_{\Delta}(\R^2 \setminus \Sigma) \cap H^1(\R^2)$ there holds $\gamma_D^{+} f_{+} = \gamma_D^{-} f_{-}\in H^1(\Sigma)$, and so 
	\begin{equation*}
		2 \bar{\bn} ( \gamma_D^{+} \dzbar f_{+} - \gamma_D^{-} \dzbar f_{-} ) = \gamma_N^{+} f_{+} - \gamma_N^{-} f_{-}.
\end{equation*} 
Using this in~\eqref{transmission_delta1}, we find that the condition $\Gamma_0 f + B \Gamma_1 f=0$ is equivalent to
\begin{equation*}
	\gamma_N^{+} f_{+} - \gamma_N^{-} f_{-}  = - B_{11} \gamma_D f,
\end{equation*}
which in combination with ~\eqref{delta_domain_inclusion}  implies that~\eqref{Domain_B1} is true.

Next, we prove that the discrete spectrum of $T_B$ is finite. For this, we consider the auxiliary quadratic form 
\begin{equation*}
 \widetilde{t}_B[f] := \| \nabla f \|_{L^2(\R^2)}^2 - \| B_{11} \|  \|\gamma_D f \|_{L^2(\Sigma)}^2, \quad \Dom(\widetilde{t}_B) := H^1(\R^2).
\end{equation*}
By \cite[Prop.~3.1]{BEL13} the form $\widetilde{t}_B$ is closed and bounded from below, and via the first representation theorem \cite[Thm.~VI.2.1]{kato} it is associated with a Schr\"odinger operator $\widetilde{T}_B$ with a $\delta$-interaction supported on $\Sigma$ of strength $-\|B_{11}\|$ that is formally given  by $-\Delta - \| B_{11} \| \delta_\Sigma$. By \cite[Thm.~3.14]{BLL12} the discrete spectrum of $\widetilde{T}_B$ is finite.
To proceed, define the sesquilinear form $t_B$ by 
\begin{equation*}
t_B[f] :=  \| \nabla f \|_{L^2(\R^2)}^2 + \langle B_{11} \gamma_D f , \gamma_D f \rangle_{\Sigma}, \quad \Dom(t_B) := H^1(\R^2).
\end{equation*}
As for $\widetilde{t}_B$ one shows that $t_B$ is closed and bounded from below. 
Moreover, using integration by parts and the transmission condition in \eqref{Domain_B1}, we see that 
$$\langle T_B f, g\rangle_{L^2(\R^2)}=t_B(f, g)\qquad \forall f\in\Dom(T_B),\, \forall g\in\Dom(t_B).$$ Since $\Dom(T_B)\subset\Dom(t_B)$ and $T_B$ is self-adjoint, this implies that $t_B$ is the quadratic form associated with $T_B$ via the first representation theorem.
Furthermore, $t_B \geq \widetilde{t}_B$ in the sense of quadratic forms and $\sigma_{\textup{ess}}(\widetilde{T}_B)=\sigma_{\textup{ess}}(T_B)=[0,+\infty)$. Since $\sigma_\textup{disc}(\widetilde{T}_B)$ is finite, this and the min-max principle imply that also $\sigma_\textup{disc}(T_B)$ has to be finite. This finishes the proof of item~(iii).

	It remains to prove item~(iv). Let $w <0$ be fixed. As $B$ is non-negative, we can apply Theorem~\ref{TBSelfadjoint}~(i) with $B_1=B_2=B^{1/2}$ and get that
	\begin{equation} \label{EquivalenceBSPrinciple}
		\begin{split}
			w \in \sigma_{\textup{p}}(T_B) \quad 
			&\Longleftrightarrow \quad 0 \in \sigma_{\textup{p}} \left(\frac{1}{\sqrt{|w|}} I + \frac{1}{\sqrt{|w|}} B^{1/2} M(w) B^{1/2} \right).
		\end{split}
	\end{equation}
	Hence, it is sufficient to study  for $w <0$ the positive eigenvalues of the compact and self-adjoint operator
	\begin{equation*}
		A(w) :=- \frac{1}{\sqrt{|w|}} B^{1/2} M(w) B^{1/2}  = - \frac{1}{\sqrt{|w|}} B^{1/2} \begin{pmatrix}
			S(w) & W(w) \\
			W(\overline{w})^{\ast} & \frac{w}{4} S(w)
		\end{pmatrix}  B^{1/2},
	\end{equation*}
	where $S(w)$ and $W(w)$ are defined by~\eqref{def_single_layer_boundary_integral_operator} and~\eqref{def_L}, respectively.
	 
	For $w<0$, denote by $(\mu_k(w))_{k=1}^{N(w)},\, N(w)\leq+\infty,$ the decreasing sequence of non-negative eigenvalues of $A(w)$, taking their multiplicities into account. If the sequence is finite, we will always extend it by zeros.  Since $\mathbb{C} \setminus [0, +\infty) \ni w \mapsto M(w)$ is a holomorphic operator valued function due to Proposition \ref{Properties_Gamma_Weyl}, Lemma~\ref{lem_convergence_ev} implies that the mapping  $w\mapsto\mu_k(w)$ is continuous.
	Furthermore, since $B^{1/2}$ is compact, we get with Corollary~\ref{cor_SK_convergence} and Proposition~\ref{lem_boundedness_T} that 
	\begin{equation*}
		\lim_{w \rightarrow -\infty} A(w) = \frac{1}{8} B^{1/2}\begin{pmatrix} 0 & 0 \\ 0 & I \end{pmatrix} B^{1/2} =: A(-\infty)
	\end{equation*}
	in the operator norm, where $I$ denotes the identity operator in $L^2(\Sigma)$. Note that the well-known identity $\sigma(CD) \setminus \{ 0 \} = \sigma(DC) \setminus \{ 0 \}$ for two bounded operators $C, D$ in a Hilbert space (see, e.g., \cite[Prop.~2.1.8]{P94}) implies that the positive eigenvalues of $A(-\infty)$ are exactly the positive eigenvalues of $\frac{1}{8}B_{22}$, that will be denoted by $\lambda_k$. Then, we conclude with Lemma~\ref{lem_convergence_ev} that
	\begin{equation*} 
		\lim_{w \rightarrow -\infty} \mu_k(w) = \lambda_k > 0.
	\end{equation*}
	Moreover, as $A(-1)$ is compact, there exists $k_0 \in \mathbb{N}$ such that $\mu_k(-1) <1 $ for all $k > k_0$. Hence, for every $k > k_0$ the function $f_k: (-\infty,-1] \to \R$ defined by 
	\begin{equation*}
		f_k(w) = \frac{1}{\sqrt{|w|}} - \mu_k(w)
	\end{equation*}
	satisfies $f_k(-1) > 0$  and converges for $w \to -\infty$ to the negative number $-\lambda_k$.
	In particular, because of the continuity of $\mu_k(w)$, there exists a  $w_k \in (-\infty,-1]$ with $f_k(w_k) = 0$ for every $k>k_0$, implying $w_k \in \sigma_{\textup{p}}(T_B)$ due to \eqref{EquivalenceBSPrinciple}. Since $\sigma_{\text{ess}}(T_B) = [0, +\infty)$ by~(i), we conclude with Theorem~\ref{TBSelfadjoint}~(i) that the points $w_k$ can not have an accumulation point in $(-\infty, -1]$ and thus, the  sequence   $(w_k)_{k \in \mathbb{N}}$ has to tend to $-\infty$ for $k \rightarrow \infty$. This finishes the proof of the theorem.
\end{proof}

\section{The Dirac operator with a non-local singular potential and its non-relativistic limit} \label{section_nrl}

In this section, we will consider the non-relativistic limit of the relativistic model of  non-local $\delta$-shell interactions, which was discussed in the recent paper \cite{HT23}. It turns out that this non-relativistic limit is of the form $T_B$ as in~\eqref{eq:def_TB} for a special parameter $B$.

To introduce the mentioned Dirac operators with non-local singular interactions, some notations have to be fixed first. Let
\begin{equation*}
\sigma_1 = \begin{pmatrix} 
0 & 1\\
1 & 0
\end{pmatrix}, \quad \sigma_2 = \begin{pmatrix} 
0 & -i\\
i & 0
\end{pmatrix} \quad \text{and} \quad \sigma_3 = \begin{pmatrix} 
1 & 0\\
0 & -1
\end{pmatrix},
\end{equation*}
be the Pauli spin matrices. We will use for $x=(x_1,x_2) \in \mathbb{C}^2$ the notation
\begin{equation*}
	\sigma \cdot x := \sigma_1 x_1 + \sigma_2 x_2 \quad \text{and} \quad \sigma \cdot \nabla  := \sigma_1 \partial_1 + \sigma_2 \partial_2.
\end{equation*}
We will make use of the function spaces 
	\begin{equation*}
		\begin{split}
			H^{1/2}_\sigma(\Omega_\pm) &=\{ f \in H^{1/2}(\Omega_\pm; \mathbb{C}^2)\mid (\sigma\cdot\nabla)f \in L^2(\Omega_\pm; \mathbb{C}^2)\} \\
			&= \left\{ f = \begin{pmatrix} f_1 \\ f_2 \end{pmatrix} \in L^2(\Omega_\pm; \mathbb{C}^2) \,\Big|\, f_1 \in H_{\partial_{\overline{z}}}^{1/2}(\Omega_\pm), f_2 \in H_{\partial_{z}}^{1/2}(\Omega_\pm) \right\},
		\end{split}
\end{equation*}
where the spaces $H_{\partial_{z}}^{1/2}(\Omega_\pm)$ and $H_{\partial_{\bar z}}^{1/2}(\Omega_\pm)$ are defined in \eqref{eq_Hdz} and \eqref{eq_Hdz_bar}, respectively. Taking Lemma \ref{Trace_Extension} into account, we see that $f \in H^{1/2}_\sigma(\Omega_\pm)$ possesses a Dirichlet trace $\gamma_D^\pm f \in L^2(\Sigma; \mathbb{C}^2)$.

Next, given $F,\, G \in L^2(\Sigma; \mathbb{C}^{2 \times 2})$, define the operator 
	\begin{equation} \label{DefNonLocalDirac}
		\begin{split}
			D_{F,G}^c f &= \left(-ic(\sigma \cdot \nabla) +\sigma_3 \frac{c^2}{2} \right) f_+ \oplus \left(-ic(\sigma \cdot \nabla) +\sigma_3 \frac{c^2}{2} \right) f_-, \\
			\Dom(D_{F,G}^c) &= \bigg\{f \in H^{1/2}_\sigma(\Omega_+)\oplus H^{1/2}_\sigma(\Omega_-) \: \big| \\
			&\quad \qquad -i(\sigma \cdot n)(\gamma_D^+ f_+ - \gamma_D^- f_-)=\frac{1}{2}F\int_\Sigma G^\ast (\gamma_D^+f_+ + \gamma_D^-f_-)\dd \sigma \bigg\}.
		\end{split}
\end{equation}
Comparing this operator to \cite[Def.~3.0.1]{HT23}, it turns out that $D_{F,G}^c$ is the rigorous mathematical operator in $L^2(\mathbb{R}^2; \mathbb{C}^2)$ associated with the formal expression 
\begin{equation*}
	\mathcal{D}_{F,G}^c=-ic\sigma \cdot \nabla + \sigma_3 \frac{c^2}{2} + c|F\delta_\Sigma\rangle\langle G\delta_\Sigma|.
\end{equation*}
More concretely, $D_{F,G}^c= c D_{F,G}$, where $D_{F,G}$ is defined in \cite[Def.~3.0.1]{HT23} with the $c$-dependent mass term $m=\frac{c}{2}$ and dimension $n=2$. Note that for $F = G = 0$ the operator $D_{0,0}^c=:D_0^c$ coincides with the free Dirac operator  
	\begin{equation*}
		D_0^c f = \left(-ic(\sigma \cdot \nabla) +\sigma_3 \frac{c^2}{2} \right) f, \quad \Dom(D_0^c) = H^1(\mathbb{R}^2; \mathbb{C}^2).
\end{equation*}

Now, we will discuss  basic properties of $D_{F,G}^c$ and, in particular, a useful resolvent formula. For this, some further objects have to be introduced. Recall that for $\widetilde{w} \in \mathbb{C} \setminus [0, +\infty)$ the operators $SL(\widetilde{w}),\, WL(\widetilde{w}),\, \widetilde{WL}(\widetilde{w}): L^2(\Sigma) \rightarrow L^2(\mathbb{R}^2)$ are defined by~\eqref{def_single_layer_potential}, \eqref{def_Psi}, and~\eqref{def_Xi}, respectively. Then, we define for $w \in \mathbb{C} \setminus ((-\infty, -\frac{c^2}{2}] \cup [\frac{c^2}{2}, +\infty))$ the operator $\Phi^c(w): L^2(\Sigma; \mathbb{C}^2) \rightarrow L^2(\mathbb{R}^2; \mathbb{C}^2)$ by
\begin{equation} \label{def_Phi_w}
  \Phi^c(w) := \begin{pmatrix} (\tfrac{w}{c} + \tfrac{c}{2}) SL(\tfrac{w^2}{c^2} - \tfrac{c^2}{4}) & -2i WL(\tfrac{w^2}{c^2} - \tfrac{c^2}{4}) \\ -2i \widetilde{WL}(\tfrac{w^2}{c^2} - \tfrac{c^2}{4}) & (\tfrac{w}{c} - \tfrac{c}{2}) SL(\tfrac{w^2}{c^2} - \tfrac{c^2}{4}) \end{pmatrix}.
\end{equation}
Furthermore, recall that for $\widetilde{w} \in \mathbb{C} \setminus [0, +\infty)$ the operators $S(\widetilde{w}), W(\widetilde{w}): L^2(\Sigma) \rightarrow L^2(\Sigma)$ are defined by~\eqref{def_single_layer_boundary_integral_operator} and~\eqref{def_L}, respectively. Then, we define for $w \in \mathbb{C} \setminus ((-\infty, -\frac{c^2}{2}] \cup [\frac{c^2}{2}, +\infty))$ the operator $\mathcal{C}^c(w): L^2(\Sigma; \mathbb{C}^2) \rightarrow L^2(\Sigma; \mathbb{C}^2)$ by
\begin{equation} \label{def_C_w}
  \mathcal{C}^c(w) := \begin{pmatrix} (\tfrac{w}{c} + \tfrac{c}{2}) S(\tfrac{w^2}{c^2} - \tfrac{c^2}{4}) & -2i W(\tfrac{w^2}{c^2} - \tfrac{c^2}{4}) \\ 2i W(\tfrac{\overline{w}^2}{c^2} - \tfrac{c^2}{4})^* & (\tfrac{w}{c} - \tfrac{c}{2}) S(\tfrac{w^2}{c^2} - \tfrac{c^2}{4}) \end{pmatrix}.
\end{equation}
Finally, for $A\in  L^2(\Sigma; \mathbb{C}^{2\times 2})$, we define a bounded operator $\Pi_A:\, L^2(\Sigma; \mathbb{C}^{2})\to \mathbb{C}^{2}$ as
\begin{equation*}
\Pi_A f:=\int_\Sigma (A^\ast f) \dd\sigma.
\end{equation*}
Its adjoint $\Pi_A^\ast:\, \mathbb{C}^{2}\to L^2(\Sigma; \mathbb{C}^{2})$ acts as multiplication by the matrix $A$. Therefore, the rank of $\Pi_A^\ast$ equals to the number of linearly independent columns of $A$ in $L^2(\Sigma; \mathbb{C}^{2})$. Note that with a constant $2\times 2$ matrix $C$ one gets $\Pi_{CA}=\Pi_A C^\ast$  and $\Pi_{CA}^\ast=C\Pi_A^\ast$. Furthermore, comparing to the \emph{bra-ket} notation, we obtain $|F\rangle\langle G|=\Pi_F^\ast\Pi_G$.

The following results are basically modifications of similar observations from \cite{HT23}.

 \begin{proposition} \label{proposition_Dirac_nonlocal}
  If $\Pi_F^\ast\Pi_G=\Pi_G^\ast\Pi_F$, then the operator $D_{F,G}^c$ defined by~\eqref{DefNonLocalDirac} is self-adjoint in $L^2(\mathbb{R}^2; \mathbb{C}^2)$ and the following is true:
  \begin{itemize}
    \item[\textup{(i)}] $\sigma_{\textup{ess}}(D_{F,G}^c) = (-\infty, -\frac{c^2}{2}] \cup [\frac{c^2}{2}, +\infty)$.
    \item[\textup{(ii)}] $\sigma_{\textup{disc}}(D_{F,G}^c)$ consists of at most two eigenvalues taking multiplicities into account.
    \item[\textup{(iii)}] For $ w \in \rho(D_{F,G}^c)$, the matrix $(I + \Pi_G  \mathcal{C}^c(w) \Pi_F^\ast)$ is invertible and 
    \begin{equation*}
      \begin{split}
        \big( D_{F,G}^c &- w \big)^{-1}  = \big( D_0^c - w \big)^{-1}  \\
        &- \frac{1}{\sqrt{c}}\Phi^c(w) \Pi_F^\ast \big(I + \Pi_G  \mathcal{C}^c(w) \Pi_F^\ast \big)^{-1} \frac{1}{\sqrt{c}}\Pi_G \Phi^c(\overline{w})^\ast 
      \end{split}
    \end{equation*}
    holds on $L^2(\mathbb{R}^2; \mathbb{C}^2)$.
  \end{itemize}
\end{proposition}
\begin{proof} 
	Using that $D_{F,G}^c= c D_{F,G}$, where $D_{F,G }$ is the operator studied in~\cite{HT23} with $m=\frac{c}{2}$, items~(i) and~(ii) follow from \cite[Thms.~3.0.1\&3.0.2]{HT23}. To get the claimed resolvent formula, we use the generalized boundary triple $\{L^2(\Sigma; \C^2), \Gamma_0^{(1/2)}, \Gamma_1^{(1/2)} \}$ introduced in \cite[Thm. 4.3]{BHSLS23} (for $s=\frac{1}{2}$ and $m=\frac{c}{2}$), which was also used in \cite{HT23} to define and study $D_{F,G}$. First, we note that by~\eqref{def_single_layer_potential},~\eqref{def_Psi}, and~\eqref{def_Xi} the map $\Phi^c(w)$ defined in~\eqref{def_Phi_w} is an integral operator with integral kernel
	\begin{equation*}
	  \begin{split}
    \frac{1}{2 \pi} \sqrt{\tfrac{w^2}{c^2} - \tfrac{c^2}{4}} K_1 &\left( - i \sqrt{\tfrac{w^2}{c^2} - \tfrac{c^2}{4}} | x-y_\Sigma |\right)  \frac{(\sigma \cdot (x-y_\Sigma))}{ | x-y_\Sigma | } \\
    &\qquad \qquad + \frac{1}{2 \pi} K_0 \left(- i \sqrt{\tfrac{w^2}{c^2} - \tfrac{c^2}{4}} | x-y_\Sigma |\right) \left( \frac{w}{c} I + \frac{c}{2} \sigma_3\right).
  \end{split}
	\end{equation*}
	Therefore, by \cite[Thm. 4.3~(ii)]{BHSLS23} the $\gamma$-field corresponding to the generalized boundary triple $\{L^2(\Sigma; \C^2), \Gamma_0^{(1/2)}, \Gamma_1^{(1/2)} \}$ coincides with  $\Phi^c(c\,\,\cdot\,)$. By a similar line of reasoning, due to \cite[Thm. 4.3~(iii)]{BHSLS23}, the Weyl function corresponding to the generalized boundary triple $\{L^2(\Sigma; \C^2), \Gamma_0^{(1/2)}, \Gamma_1^{(1/2)}\}$ coincides with  $\mathcal{C}^{c}(c\,\,\cdot\,)$ defined in \eqref{def_C_w}. Next, for any $f \in H^{1/2}_\sigma(\Omega_+)\oplus H^{1/2}_\sigma(\Omega_-)$ one has 
		\begin{equation*}
			f \in \Dom(D_{F,G}) \quad \Longleftrightarrow \quad \Gamma_0^{(1/2)} f +\Pi_F^\ast\Pi_G \Gamma_1^{(1/2)} f = 0;
		\end{equation*}
		cf. \cite[Def. 3.0.1]{HT23}. Therefore, this and Theorem~\ref{GBT_TB_SA}~(iii) applied with $B_1 = \Pi_F^*$ and $B_2 = \Pi_G$ imply 
		\begin{equation*}
			\begin{split}
				(D_{F,G}^c-w)^{-1}&=\frac{1}{c} \big(D_{F,G}-\frac{w}{c} \big)^{-1} \\
				&=\frac{1}{c} \big( D_{0,0} - \frac{w}{c} \big)^{-1}- \frac{1}{c}\Phi^c(w) \Pi_F^\ast \big(I + \Pi_G  \mathcal{C}^c(w) \Pi_F^\ast \big)^{-1} \Pi_G \Phi^c(\overline{w})^\ast. 
			\end{split}
		\end{equation*}
\end{proof}

In order to find the non-relativistic limit of $D_{F,G}^c$, one has to subtract the particle's rest energy $\frac{c^2}{2}$, which is a completely relativistic object, from the total energy and then compute the limit of the resolvent of the resulting operator $D_{F,G}^c - c^2/2$, as $c \rightarrow +\infty$. To get a non-trivial limit, we will in fact work with the operator $D_{F_c,G_c}^c $,
where
\begin{equation} \label{TransformationF}
F_c:=S_c F,\, G_c:=S_c G\quad\text{with}\quad S_c:=\begin{pmatrix}
\frac{1}{\sqrt{c}} & 0\\
0 & \sqrt{c}
\end{pmatrix},
\end{equation}
instead of $D_{F,G}^c$ itself, cf. \cite{HT22,BD94} for similar rescalings.
Define for  $w \in \mathbb{C} \setminus \mathbb{R}$ the numbers $\widetilde{w}_c :=w+\frac{c^2}{2}$ and  $w_c := w+\frac{w^2}{c^2}$. 
Then, taking Proposition~\ref{proposition_Dirac_nonlocal}~(iii) into account, we have to study the convergence of the operator
\begin{equation} \label{equation_krein_Dirac} 
  \begin{split}
        \big( D_{F_c,G_c}^c &- \widetilde{w}_c \big)^{-1} = \big( D_{0}^c - \widetilde{w}_c \big)^{-1} \\
        &- \frac{1}{\sqrt{c}}\Phi^c(\widetilde{w}_c)S_c \Pi_F^\ast \big(I + \Pi_G S_c\mathcal{C}^c(\widetilde{w}_c)S_c\Pi_F^\ast \big)^{-1} \frac{1}{\sqrt{c}}\Pi_G S_c\Phi^c(\overline{\widetilde{w}_c})^\ast
      \end{split}
\end{equation}
as $c \rightarrow \infty$. Note that 
\begin{equation} \label{eq.-left side of D.resolvent}
\begin{split}
&\frac{1}{\sqrt{c}}\Phi^c(\widetilde{w}_c) S_c  =\begin{pmatrix}
\left(1+\frac{w}{c^2}\right)  SL(w_c) & -2i  WL(w_c) \\ 
-\frac{2i}{c} \widetilde{WL}(w_c) & \frac{w}{c}  SL(w_c) \end{pmatrix} 
\end{split}
\end{equation}
and
\begin{equation} \label{def_N}
S_c\mathcal{C}^c(\widetilde{w}_c)S_c=\begin{pmatrix} (1+\tfrac{w}{c^2}) S(w_c) & -2i W(w_c) \\ 2i W(\overline{w_c})^* & w S(w_c) \end{pmatrix}=:N_c(w).
\end{equation}
We will find the limits of \eqref{eq.-left side of D.resolvent} and \eqref{def_N} in the respective uniform operator topologies as $c\to+\infty$. For this, we need the following preliminary lemma.

\begin{lemma}\label{lemma-estimate of K_j}
Let $w \in \mathbb{C} \setminus \mathbb{R}$, $w_c = w + w^2/c^2$, and $t_1, t_2, t_3 , t_4 : \mathbb{R}^2 \setminus \{0\} \to \mathbb{C}$ be defined by
\begin{equation*}
\begin{split}
t_1(x)&=\frac{1}{ c}\sqrt{w_c}K_1(-i\sqrt{w_c}|x|)\\ 
t_2(x)&=  \frac{1}{c^2}K_0(-i\sqrt{w_c}|x|),\\
t_3(x)&= K_0(-i\sqrt{w_c}|x|)-K_0(-i\sqrt{w}|x|),\\
t_4(x) &= \sqrt{w_c} K_1(-i\sqrt{w_c}|x|) -\sqrt{w} K_1(-i\sqrt{w}|x|).
\end{split}
\end{equation*}
Then there exist positive constants $\kappa_1,\kappa_2, \kappa_3, \kappa_4$, depending only on $w$, such that for all sufficiently large $c>0$ the estimates
\begin{equation*}
|t_1(x)|\leq \frac{\kappa_1}{c}(|x|^{-1}+1)\ee^{-\kappa_2 |x|}
\end{equation*}
and 
\begin{equation*}
|t_j(x)|\leq \frac{\kappa_1}{c^2}(|\ln|x||+1)\ee^{-\kappa_2 |x|} \leq \frac{\kappa_3}{c^2}(|x|^{-1/4}+1)\ee^{-\kappa_4 |x|}, \qquad j \in \{ 2,3,4 \},
\end{equation*}
hold on $\R^2\setminus\{0\}$.
\end{lemma}

\begin{proof}
First, we mention that the second estimate on $|t_j|$, $j \in \{ 2, 3, 4\}$, always follows from the first one, as soon as this one is established.
The estimates for the functions $t_1$ and $t_3$ follow from the proof of \cite[Lem.~3.1]{BHS23}, see equation~(3.8) and the consideration on the function denoted by $t_2$ in \cite{BHS23}. The estimate for $t_2$ can be deduced from Lemma~\ref{lemma_Bessel_functions}~(iv). Hence it remains to control $t_4$.  Using the expansion of $K_1$ from Lemma~\ref{lemma_Bessel_functions}~(iii), one gets
\begin{equation}  \label{DiffK1}
\begin{split}
\sqrt{w_c}&K_1(-i\sqrt{w_c}|x|)-\sqrt{w} K_1(-i\sqrt{w}|x|) \\
&= - i w_c |x| g_1(-w_c |x|^2) \ln(-i \sqrt{w_c} |x|) + i w |x| g_1(-w|x|^2)  \ln(-i \sqrt{w} |x|) \\
&\quad - i w_c|x| g_2(-w_c |x|^2) +i w|x| g_2(-w |x|^2).
\end{split}
\end{equation}
Due to the holomorphy of $g_1$ and $g_2$, there exists for any $R>0$ a constant $C>0$ that is independent of $c$ such that
\begin{equation*}
|x|\big|  g_j(-w_c |x|^2) -  g_j(-w |x|^2) \big| \leq |w-w_c| C = \frac{|w|^2 C}{c^2} \quad \forall x \in B(0,R)
\end{equation*}
is valid for sufficiently large $c>0$. Similarly, as we chose the complex square root such that $\text{Im}(\sqrt{w})>0$, the holomorphy of the logarithm on the set $\mathbb{C} \setminus (- \infty,0]$ yields $\forall x \in B(0,R)\setminus\{0\}$,
\begin{equation*}
\big| \ln(-i \sqrt{w_c} |x|)  - \ln(-i \sqrt{w} |x|) \big|=\big| \ln(-i \sqrt{w_c})  - \ln(-i \sqrt{w}) \big|\leq\frac{C}{c^2}. 
\end{equation*}
In combination with \eqref{DiffK1}, this yields 
\begin{equation*}
\big|\sqrt{w_c}K_1(-i\sqrt{w_c}|x|)-\sqrt{w} K_1(-i\sqrt{w}|x|)\big| \leq \frac{C}{c^2}(|\ln|x||+1) \quad \forall x \in B(0,R)\setminus\{0\}
\end{equation*}
for all sufficiently large $c>0$. To complete the proof, note that  Lemma~\ref{lemma_Bessel_functions}~(v) implies that there exist constants $C, \kappa > 0$ which are independent of $c$ such that for all sufficiently large $c>0$
\begin{equation} \label{DiffK1LargeOrder1}
\left| \sqrt{w_c} - \sqrt{w} \right| | K_1(-i\sqrt{w_c}|x|)| \leq \frac{C}{c^2} \ee^{-\kappa |x|} \quad \forall |x|>R.
\end{equation}
Moreover, by the fundamental theorem of calculus one has for all $|x| > R$ and all sufficiently large $c>0$
\begin{equation*}
| K_1(-i\sqrt{w_c}|x|) - K_1(-i\sqrt{w}|x|) | \leq \int_0^1 \left| \frac{\dd}{\dd t} K_1\left(-i\sqrt{w+ t \frac{w^2}{c^2}} |x| \right) \right| \dd t  \leq \frac{C}{c^2} \ee^{- \kappa |x|},
\end{equation*}
where Lemma~\ref{lemma_Bessel_functions}~(v) was used.
Combining this with \eqref{DiffK1LargeOrder1} leads to 
\begin{equation*}
\big|\sqrt{w_c}K_1(-i\sqrt{w_c}|x|)-\sqrt{w} K_1(-i\sqrt{w}|x|)\big| \leq \frac{C}{c^2} \ee^{- \kappa |x|} \quad \forall |x|>R,
\end{equation*}
which finishes the proof of the lemma.
\end{proof}

With the help of Lemma~\ref{lemma-estimate of K_j} we can study the convergence of the terms in~\eqref{eq.-left side of D.resolvent}. Recall that for $w \in \mathbb{C} \setminus \mathbb{R}$ the operators $SL(w)$, $WL(w)$ and $\Phi^c(w)$ are defined by \eqref{def_single_layer_potential}, \eqref{def_Psi} and \eqref{def_Phi_w},  respectively.

\begin{lemma}\label{lemma-parts of the resolvent}
Let $w\in\mathbb{C}\setminus\mathbb{R}$. Then there exists a constant $\kappa > 0$  depending only on $w$ and $\Sigma$ such that
\begin{equation} \label{convergence_Phi}
\left\lVert \frac{1}{\sqrt{c}}\Phi^c(\widetilde{w}_c) S_c - \begin{pmatrix}
SL(w) & -2i WL(w)\\
0 & 0
\end{pmatrix}\right\rVert \leq \frac{\kappa}{c},
\end{equation}
and
\begin{equation} \label{convergence_Phi_star} 
\left\lVert\frac{1}{\sqrt{c}}S_c\Phi^c(\overline{\widetilde{w}_c})^\ast -\begin{pmatrix}
SL(\overline{w})^\ast & 0\\
2i WL(\overline{w})^\ast & 0
\end{pmatrix} \right\rVert \leq \frac{\kappa}{c}.
\end{equation}
\end{lemma}

\begin{proof}  
First, putting together the integral representation of $SL(w)$ in~\eqref{def_single_layer_potential}, the estimate of $t_3$ in Lemma~\ref{lemma-estimate of K_j}, and Proposition~\ref{prop.-pot.int.op.} (for $\theta = 1/4$) we deduce that 
\begin{equation*}
\|SL(w_c) - SL(w) \| =\mathcal{O}(c^{-2}).
\end{equation*}
Similarly, using the integral representation of $WL(w)$ in~\eqref{def_Psi}, the estimate of $t_4$ in Lemma~\ref{lemma-estimate of K_j}, and again Proposition~\ref{prop.-pot.int.op.} (for $\theta = 1/4$), we arrive at
\begin{equation*}
\| WL(w_c) - WL(w) \|=\mathcal{O}(c^{-2}).
\end{equation*}
Furthermore, the estimates for $t_1$ and $t_2$ in Lemma \ref{lemma-estimate of K_j}  yield with Proposition~\ref{prop.-pot.int.op.} (for $\theta = 1/4$) that 
\begin{equation*}
\left\|\frac{1}{c}SL(w_c)\right\|=\mathcal{O}(c^{-1}) \quad \text{and} \quad \left\|\frac{1}{c} \widetilde{WL}(w_c) \right\|=\mathcal{O}(c^{-1}).
\end{equation*}
Now, the estimate in~\eqref{convergence_Phi} follows immediately from \eqref{eq.-left side of D.resolvent}.  By hermitian conjugation, also~\eqref{convergence_Phi_star} is true.
\end{proof}

The next lemma is devoted to the convergence of $N_c(w)$ defined by~\eqref{def_N}. Recall that $M(w)$ denotes the Weyl function associated with the generalized boundary triple $\{ L^2(\Sigma; \mathbb{C}^2), \Gamma_0, \Gamma_1 \}$; cf. Proposition~\ref{IsQBT}.

\begin{lemma}\label{lemma-matrix in the resolvent}
Let  $w\in\mathbb{C}\setminus\mathbb{R}$ and set $V = \diag (1,-2i) \in \mathbb{C}^{2 \times 2}$. Then there exists a constant $\kappa >0$ depending only on $w$ and $\Sigma$  such that for all sufficiently large $c>0$
\begin{equation*}
\left\lVert N_c(w)-V^* M(w) V  \right\rVert\leq \frac{\kappa}{c^2}.
\end{equation*}
\end{lemma}

\begin{proof}
First, taking Proposition~\ref{IsQBT}~(iii), Lemma~\ref{lemma_W}, and~\eqref{IntegralRepresSLBIO} into account, one sees that $M(w)$ is a strongly singular boundary integral operator with integral kernel $Q_1(x_\Sigma-y_\Sigma)$, where 
\begin{equation*}
Q_1(x) =  \begin{pmatrix}
\frac{1}{2 \pi} K_0(-i\sqrt{w} |x|) &  \frac{i \sqrt{w}}{4 \pi} \frac{\overline{\bf{x}}}{|x|}K_1( - i \sqrt{w} |x|) \\
 \frac{- i \sqrt{w}}{4 \pi} \frac{\bf{x}}{|x|}K_1( - i \sqrt{w} |x|) & \frac{w}{8 \pi} K_0(-i\sqrt{w} |x|)
\end{pmatrix}.
\end{equation*}
Hence, an explicit calculation yields 
\begin{equation*}
V^* Q_1(x) V = \begin{pmatrix}
\frac{1}{2 \pi} K_0(-i\sqrt{w} |x|) &   \frac{ \sqrt{w}}{2 \pi} \frac{\overline{\bf{x}}}{|x|}K_1( - i \sqrt{w} |x|) \\
\frac{\sqrt{w}}{2 \pi} \frac{\bf{x}}{|x|}K_1( - i \sqrt{w} |x|) & \frac{w}{2 \pi} K_0(-i\sqrt{w} |x|)
\end{pmatrix}.
\end{equation*}
In a similar vein, $N_c(w)$ is a strongly singular boundary integral operator with kernel $Q_2(x_\Sigma-y_\Sigma)$, where
\begin{equation*}
Q_2(x) = \begin{pmatrix}
\frac{1}{2 \pi} \big(1+\frac{w}{c^2}\big) K_0(-i\sqrt{w_c} |x|) &   \frac{ \sqrt{w_c}}{2 \pi} \frac{\overline{\bf{x}}}{|x|}K_1( - i \sqrt{w_c} |x|) \\
\frac{\sqrt{w_c}}{2 \pi} \frac{\bf{x}}{|x|}K_1( - i \sqrt{w_c} |x|) & \frac{w}{2 \pi} K_0(-i\sqrt{w_c} |x|)
\end{pmatrix}.
\end{equation*}
Therefore, $N_c(w)-V^* M(w) V$ is an integral operator with kernel $H (x_\Sigma - y_\Sigma) := Q_2 (x_\Sigma - y_\Sigma) - Q_1 (x_\Sigma - y_\Sigma)$ with entries 
\begin{equation*}
\begin{split}
H_{11}(x) &= \frac{1}{2 \pi} \big(1+\frac{w}{c^2}\big) K_0(-i\sqrt{w_c} |x|) -  \frac{1}{2 \pi}  K_0(-i\sqrt{w} |x|), \\
H_{12}(x) &=  \frac{ \sqrt{w_c}}{2 \pi} \frac{\overline{\bf{x}}}{|x|}K_1( - i \sqrt{w_c} |x|) - \frac{ \sqrt{w}}{2 \pi} \frac{\overline{\bf{x}}}{|x|}K_1( - i \sqrt{w} |x|), \\
H_{21}(x) &= \frac{\sqrt{w_c}}{2 \pi} \frac{\bf{x}}{|x|}K_1( - i \sqrt{w_c} |x|) - \frac{\sqrt{w}}{2 \pi} \frac{\bf{x}}{|x|}K_1( - i \sqrt{w} |x|), \\
H_{22}(x) &=  \frac{w}{2 \pi} K_0(-i\sqrt{w_c} |x|) - \frac{w}{2 \pi} K_0(-i\sqrt{w} |x|).
\end{split}
\end{equation*}
It follows from the estimates for $t_2,\, t_3$, and $t_4$ from Lemma \ref{lemma-estimate of K_j}  that there exist constants $\kappa_3, \kappa_4 >0$ depending only on $w$ and $\Sigma$ such that for all sufficiently large $c>0$  and $i,\, j\in\{1,2\}$
\begin{equation*} 
|H_{ij}(x)| \leq \frac{\kappa_3}{c^2} (|x|^{-1/4}+1)\ee^{-\kappa_4 |x|}\qquad \forall x\in\R^2\setminus\{0\}
\end{equation*}
is valid. Consequently, Proposition \ref{prop.-sing.int.op.on Sigma} implies the claimed operator norm estimate.
\end{proof}

Finally, we will need the following elementary observation. 

\begin{lemma} \label{lemma_composition}
Let $a,\, b>0$, $\mathscr{X}_{i}$, $i\in\{1,2,3\}$, be Banach spaces and $(A_n)$ and $(B_n)$ stand for sequences of bounded operators from $\mathscr{X}_2$ to $\mathscr{X}_3$ and from $\mathscr{X}_1$ to $\mathscr{X}_2$ that converge to $A$ and $B$ in the respective operator norms with the convergence rate $\mathcal{O}(n^{-a})$ and $\mathcal{O}(n^{-b})$, respectively. Then
$$\|A_n B_n-AB\|_{\mathscr{X}_1\to\mathscr{X}_3}=\mathcal{O}(n^{-\min\{a,b\}})\quad\text{ as }n\to+\infty.$$
\end{lemma}
\begin{proof}
By the triangle inequality, we get
\begin{equation*}
\begin{split}
\|A_n B_n-AB\|&_{\mathscr{X}_1\to\mathscr{X}_3}\\
&\leq\|A_n\|_{\mathscr{X}_2\to\mathscr{X}_3}\|B_n-B\|_{\mathscr{X}_1\to\mathscr{X}_2}+\|A_n-A\|_{\mathscr{X}_2\to\mathscr{X}_3}\|B\|_{\mathscr{X}_1\to\mathscr{X}_2}.
\end{split}
\end{equation*}

\end{proof}

Now, we are prepared to find the non-relativistic limit of $D_{F_c,G_c}^c$.

\begin{theorem}\label{theorem-nr lim}
Let $F,\, G\in L^2(\Sigma;\mathbb{C}^{2\times 2})$ be such that $\Pi_F^\ast\Pi_G=\Pi_G^\ast\Pi_F$, $F_c$ and $G_c $ be defined by~\eqref{TransformationF}, and $V = \diag (1,-2i)$.
Furthermore, let $T_B$ be defined by~\eqref{eq:def_TB} with $B=\Pi_{VF}^\ast\Pi_{VG}$ and $D_{F_c,G_c}^c$ be as in \eqref{DefNonLocalDirac}. Then  for every $w \in \mathbb{C} \setminus \mathbb{R}$  there exists a constant $\kappa > 0$ such that for all $c > 0$ sufficiently large
\begin{equation*}
\left\lVert\left(D_{F_c,G_c}^c-w -\frac{c^2}{2} \right)^{-1}-(T_B-w)^{-1}\otimes \begin{pmatrix}
1 & 0\\
0 & 0
\end{pmatrix}\right\rVert \leq \frac{\kappa}{c}.
\end{equation*}
\end{theorem}

\begin{proof} 
Because of \cite[Lem.~3.1]{BHS23} or \cite[Cor. 6.2]{Thaller} one finds a constant $C>0$ such that for all sufficiently large $c>0$,
\begin{equation} \label{eq_res_free_nonrel_lim}
\left\| \left(D_0^c-w-\frac{c^2}{2} \right)^{-1} - (-\Delta-w )^{-1} \otimes  \begin{pmatrix} 1 &0 \\ 0 & 0 \end{pmatrix} \right\| \leq \frac{C}{c}.
\end{equation}
Moreover, using Lemmas~\ref{lemma-matrix in the resolvent} and \ref{lemma_composition}, and  \cite[Thm.~IV.1.16]{kato}, we find that
\begin{equation*} 
\begin{split}
\big(I + \Pi_{G}S_c \mathcal{C}^c(\widetilde{w}_c)&S_c\Pi_{F}^\ast\big)^{-1} = \big(I + \Pi_{G}N_c(w)\Pi_{F}^\ast \big)^{-1}\\
& \stackrel{c\to+\infty}{\longrightarrow} \big(I+\Pi_{G}V^\ast M(w)V\Pi_{F}^\ast \big)^{-1}=\big(I+\Pi_{VG} M(w)\Pi_{VF}^\ast \big)^{-1} 
\end{split}
\end{equation*}
and that the rate of convergence is $\mathcal{O}(c^{-2})$. 
Substituting this together with~\eqref{convergence_Phi}, \eqref{convergence_Phi_star}, and \eqref{eq_res_free_nonrel_lim} into~\eqref{equation_krein_Dirac}, and using Lemma~\ref{lemma_composition}, we obtain that
\begin{equation*}
\begin{split}
  \lim_{c\to +\infty} & \bigg( D_{F_c,G_c}^c -  w - \frac{c^2}{2}  \bigg)^{-1} = \,(-\Delta-w )^{-1} \otimes  \begin{pmatrix} 1 &0 \\ 0 & 0 \end{pmatrix}\\
    &-\begin{pmatrix} SL(w) & WL(w)\\ 
		0 & 0
	  \end{pmatrix}\Pi_{VF}^\ast 	
  \big(I+\Pi_{VG} M(w)\Pi_{VF}^\ast \big)^{-1} \Pi_{VG} 
  	\begin{pmatrix}
   	SL(\overline{w})^\ast & 0 \\
   	WL(\overline{w})^\ast & 0
	\end{pmatrix}
\end{split}
\end{equation*}
in the operator norm with the convergence rate $\mathcal{O}(c^{-1})$. The limit can be rewritten as
\begin{equation*}
\bigg[(-\Delta-w )^{-1}  - \gamma(w) B_1 \big(I+B_2 M(w) B_1 \big)^{-1}  B_2\gamma(\overline{w})^\ast \bigg] \otimes  \begin{pmatrix} 1 &0 \\ 0 & 0 \end{pmatrix} 
\end{equation*}
where $B_1=\Pi_{VF}^\ast$ and $B_2=\Pi_{VG}$ are compact, and $\gamma(w)$ and $M(w)$ are the values of the $\gamma$-field  and the Weyl function associated with the generalized boundary triple $\{ L^2(\Sigma; \mathbb{C}^2), \Gamma_0, \Gamma_1 \}$; cf. Proposition~\ref{IsQBT}. Since the latter expression coincides exactly with $(T_B - w)^{-1} \otimes \big( \begin{smallmatrix} 1 & 0 \\ 0 & 0 \end{smallmatrix} \big)$, see Theorem~\ref{TBSelfadjoint}, the claim of the theorem follows.
\end{proof}

\subsection*{Acknowledgements}
The authors thank the Ministry of Education, Youth and Sports of the Czech Republic, the BMBWF, and the Austrian Agency for International Cooperation in Education and Research (OeAD) for support within the project 8J23AT025 and CZ 16/2023, respectively. This research was funded in part by the Austrian Science Fund (FWF) 10.55776/P 33568-N. L. Heriban acknowledges  the support by the EXPRO grant No. 20-17749X of the Czech Science Foundation (GA\v{C}R).

\end{document}